\documentclass[11pt,a4paper,british]{amsart}
\usepackage[defs,semib,cref,href,amsfix]{tms}

\title[Minimal Lagrangian connections]{Minimal Lagrangian Connections on\\ Compact Surfaces}
\author[T.~Mettler]{Thomas Mettler}
\date{20th July 2019}
\address{Institut f\"ur Mathematik, Goethe-Universit\"at Frankfurt, Frankfurt am Main, Germany}
\email{mettler@math.uni-frankfurt.de}

\begin{document}

\begin{abstract}
We introduce the notion of a minimal Lagrangian connection on the tangent bundle of a manifold and classify all such connections in the case where the manifold is a compact oriented surface of non-vanishing Euler characteristic. Combining our classification with results of Labourie and Loftin, we conclude that every properly convex projective surface arises from a unique minimal Lagrangian connection.   
\end{abstract}

\maketitle

\section{Introduction}

\subsection{Background}

A~\textit{projective manifold} is a pair $(M,\mathfrak{p})$ consisting of a smooth manifold $M$ and a~\textit{projective structure} $\mathfrak{p}$, that is, an equivalence class of torsion-free connections on the tangent bundle $TM$, where two such connections are called \textit{projectively equivalent} if they share the same geodesics up to parametrisation. A projective manifold $(M,\mathfrak{p})$ is called~\textit{properly convex} if it arises as a quotient of a  properly convex open set $\tilde{M}\subset \mathbb{RP}^n$ by a group $\Gamma\subset \mathrm{PSL}(n+1,\R)$ of projective transformations which acts discretely and properly discontinuously. The geodesics of $\mathfrak{p}$ are the projections to $M=\Gamma\!\setminus\!\tilde{M}$ of the projective line segments contained in $\tilde{M}$. In particular, locally the geodesics of a properly convex projective structure $\mathfrak{p}$ can be mapped diffeomorphically to segments of straight lines, that is, $\mathfrak{p}$ is~\textit{flat}. 

It follows from the work of Cheng--Yau~\cite{MR0437805,MR859275} that the universal cover $\tilde{M}$ of a properly convex projective manifold $(M,\mathfrak{p})$ determines a unique properly embedded hyperbolic affine sphere $f : \tilde{M} \to \R^{n+1}$, which is asymptotic to the cone over $\tilde{M}$ in $\R^{n+1}$. The~\textit{Blaschke metric} and~\textit{Blaschke connection} induced by $f$ descend to the quotient $\Gamma\!\setminus\!\tilde{M}$ and equip $M$ with a complete Riemannian metric $g$ and projectively flat connection $\nabla \in \mathfrak{p}$, see the work of Loftin~\cite{MR1828223}. The difference between $\nabla$ and the Levi-Civita connection of the Blaschke metric is encoded in terms of a cubic form, the so-called~\textit{Fubini--Pick form} of $f$. For an introduction to affine differential geometry the reader may consult~\cite{MR1311248} as well as~\cite{MR2743442} for a nice survey on affine spheres.

Properly convex projective surfaces are of particular interest, as they may be seen -- through the work of Hitchin~\cite{MR1174252}, Goldman~\cite{MR1053346} and Choi--Goldman~\cite{MR1145415} -- as the natural generalisation of the notion of a hyperbolic Riemann surface. In the case of a properly convex oriented surface $(\Sigma,\mathfrak{p})$, the Fubini--Pick form is the real part of a cubic differential that is holomorphic with respect to the Riemann surface structure on $\Sigma$ defined by the orientation and the conformal equivalence class of the Blaschke metric. Conversely, Wang~\cite{MR1178538} observed that a holomorphic cubic differential $C$ on a closed hyperbolic Riemann surface $(\Sigma,[g])$ determines a unique conformal Riemannian metric $g$ whose Gauss curvature $K_g$ satisfies 
\begin{equation}\label{eq:wang}
K_g=-1+2|C|^2_g,
\end{equation} where $|C|_g$ denotes the point-wise tensor norm of $C$ with respect to the Hermitian metric induced by $g$ on the third power of the canonical bundle of $\Sigma$. Furthermore, the pair $(g,\Re(C))$ can be realized as the Blaschke metric and Fubini--Pick form of a complete hyperbolic affine sphere $f : \tilde{M} \to \R^3$ defined on the universal cover $\tilde{M}$ of $M$. In particular, combining Wang's work with the work of Loftin establishes -- on a compact oriented surface of negative Euler characteristic -- a bijective correspondence between properly convex projective structures and pairs $([g],C)$ consisting of a conformal structure $[g]$ and a cubic holomorphic differential $C$, see~\cite{MR1828223}. This correspondence was also discovered independently by Labourie~\cite{MR2402597}.  Since then, Benoist--Hulin~\cite{MR3039771} have extended the correspondence to noncompact projective surfaces with finite Finsler volume and Dumas--Wolf~\cite{MR3432157} study the case of polynomial cubic differentials on the complex plane. 

In~\cite{MR0066020}, Libermann constructs a para-K\"ahler structure $(h_0,\Omega_0)$ on the open submanifold $A_0\subset \mathbb{RP}^n\times \mathbb{RP}^{n*}$ consisting of non-incident point-line pairs. A para-K\"ahler structure may be thought of as a split-complex analogue of the notion of a Kähler structure. In particular, $h_0$ is a pseudo-Riemannian metric of split-signature $(n,n)$ and $\Omega_0$ a symplectic form, so that there exists an endomorphism of the tangent bundle relating $h_0$ and $\Omega_0$ which squares to become the identity map. In~\cite{MR2854275,MR2854277}, Hildebrand  -- see also the related work~\cite{MR3137456,MR3303003,MR1972461} -- observed that proper affine spheres $f : M \to \R^{n+1}$ correspond to minimal Lagrangian immersions $\hat{f} : M \to A_0$. Thus, the result of Hildebrand, combined with the work of Cheng--Yau, associates a minimal Lagrangian immersion to every properly convex projective manifold.  	

\subsection{Minimal Lagrangian connections}

Here we propose a generalization of the notion of a properly convex projective surface which arises naturally from the concept of a~\textit{minimal Lagrangian connection}. In joint work with Dunajski the author has shown that the construction of Libermann is a special case of a more general result: In~\cite{MR3833818}, it is shown that a projective structure $\mathfrak{p}$ on an $n$-manifold $M$ canonically defines an almost para--K\"ahler structure $(h_{\mathfrak{p}},\Omega_{\mathfrak{p}})$ on the total space of a certain affine bundle $A\to M$, whose underlying vector bundle is the cotangent bundle of $M$. The bundle $A \to M$ has the crucial property that its sections are in one-to-one correspondence with the representative connections of $\mathfrak{p}$. Therefore, fixing a representative connection $\nabla \in \mathfrak{p}$ gives a section $s_{\nabla} : M \to A$ and hence an isomorphism $\psi_{\nabla} : T^*M \to A$, by declaring the origin of the affine fibre $A_p$ to be $s_{\nabla}(p)$ for all $p \in M$. Correspondingly, we obtain a pair $(h_{\nabla},\Omega_{\nabla})=\psi_{\nabla}^*(h_{\mathfrak{p}},\Omega_{\mathfrak{p}})$ on the total space of the cotangent bundle. Besides being a geometric structure of interest in itself (see~\cite{MR3833818} for details), the pair $(h_{\nabla},\Omega_{\nabla})$ has the natural property 
\[
o^*h_{\nabla}=(s_{\nabla})^*h_{\mathfrak{p}}=-\left(\frac{1}{n-1}\right)\mathrm{Ric}^+(\nabla) 
\]
and
\[
o^*\Omega_{\nabla}=(s_{\nabla})^*\Omega_{\mathfrak{p}}=\left(\frac{1}{n+1}\right)\mathrm{Ric}^{-}(\nabla),
\]
where $o : M \to T^*M$ denotes the zero-section and $\mathrm{Ric}^{\pm}(\nabla)$ the symmetric (respectively, the anti-symmetric) part of the Ricci curvature $\mathrm{Ric}(\nabla)$ of $\nabla$. Consequently, we call $\nabla$~\textit{Lagrangian} if the Ricci tensor of $\nabla$ is symmetric, or equivalently, if the zero-section $o$ is a Lagrangian submanifold of $(T^*M,\Omega_{\nabla}$). Likewise, we call $\nabla$~\textit{timelike/spacelike} if $\pm\mathrm{Ric}^+(\nabla)$ is positive definite, or equivalently, if the zero-section $o$ is a timelike/spacelike submanifold of $(T^*M,h_{\nabla})$. The upper sign corresponds to the timelike case and lower sign to the spacelike case. Moreover, we call $\nabla$~\textit{minimal} if the zero-section is a minimal submanifold of $(T^*M,h_{\nabla})$. 

We henceforth restrict our considerations to the case of oriented surfaces. We  show (see \cref{thm:minchar} below) that a timelike/spacelike Lagrangian connection $\nabla$ is minimal if and only if 
\[
R^{ij}\left(2\nabla_iR_{jk}-\nabla_k R_{ij}\right)=0, 
\]
where $R_{ij}$ denotes the Ricci tensor of $\nabla$ and $R^{ij}$ its inverse. We then show that a minimal Lagrangian connection $\nabla$ on an oriented surface $\Sigma$ defines a triple $(g,\beta,C)$ on $\Sigma$, consisting of a Riemannian metric $g$, a $1$-form $\beta$ and a cubic differential $C$, so that the following equations hold  

\begin{equation}\label{eq:mainintro}
K_g=\pm 1+2\,|C|_g^2+\delta_g\beta, \qquad \ov{\partial} C=\left(\beta-\i\star_g \beta\right)\otimes C,\qquad 
\d \beta=0.
\end{equation}
As usual, $\i=\sqrt{-1}$, $\ov{\partial}$ denotes the ``del-bar'' operator with respect to the integrable almost complex structure $J$ induced on $\Sigma$ by $[g]$ and the orientation, $\star_g$, $\delta_g$ and $K_g$ denote the Hodge-star, co-differential and Gauss curvature with respect to $g$, respectively. Recall that $|C|_g$ denotes the point-wise tensor norm of $C$ with respect to the Hermitian metric induced by $g$ on the third power of the canonical bundle of $\Sigma$. As a consequence, we use a result of Labourie~\cite{MR2402597} to prove that if $\nabla$ is a spacelike minimal Lagrangian connection on a compact oriented surface $\Sigma$ defining a flat projective structure $\mathfrak{p}(\nabla)$, then $(\Sigma,\mathfrak{p})$ is a properly convex projective surface. Moreover, the zero-section is a totally geodesic submanifold of $(T^*\Sigma,h_{\nabla})$ if and only if $\nabla$ is the Levi-Civita connection of a hyperbolic metric.

We also show that a minimal Lagrangian connection defines a flat projective structure if and only if $\beta$ vanishes identically. In particular, we recover Wang's equation~\eqref{eq:wang} in the projectively flat case. In the projectively flat case Labourie~\cite{MR2402597} interpreted the first two equations as an instance of Hitchin's Higgs bundle equations~\cite{MR887284}. In the case with $\beta\neq 0$ the above triple of equations falls into the general framework of~\textit{symplectic vortex equations}~\cite{MR1959059} (see also~\cite{MR1086749}). Furthermore, it appears likely that the above equations also admit an interpretation in terms of affine differential geometry, but this will be addressed elsewhere.  

 The last two of the equations~\eqref{eq:mainintro} say that locally there exists a (real-valued) function $r$ so that $\e^{-2r}C$ is holomorphic. As a consequence of this we show that the only examples of minimal Lagrangian connections on the $2$-sphere are Levi-Civita connections of metrics of positive Gauss curvature. 

Furthermore, if $(\Sigma,[g])$ is a compact Riemann surface of negative Euler characteristic $\chi(\Sigma)$, then the metric $g$ of the triple $(g,\beta,C)$ is uniquely determined in terms of $([g],\beta,C)$. This leads to a quasi-linear elliptic PDE of vortex type, which belongs to a class of equations solved in~\cite{MR3137456} using the technique of sub -- and supersolutions (see also~\cite{MR3432157} for the case when $\beta$ vanishes identically). Here instead, we use the calculus of variations and prove existence and uniqueness of a smooth minimum of the following functional defined on the Sobolev space $W^{1,2}(\Sigma)$ 
\[
\mathcal{E}_{\kappa,\tau} : W^{1,2}(\Sigma) \to \ov{\R}, \quad u\mapsto \frac{1}{2}\int_{\Sigma}|\d u|^2_{g_0}-2u-\kappa\e^{2u}+\tau\e^{-4u}dA_{g_0},
\]
where $\kappa,\tau \in C^{\infty}(\Sigma)$ satisfy $\kappa<0$, $\tau \geqslant 0$ and $g_0 $ denotes the hyperbolic metric in the conformal equivalence class $[g]$. 

An immediate consequence of~\eqref{eq:mainintro} is that the area of a spacelike minimal Lagrangian connection $\nabla$ -- by which we mean the area of $o(\Sigma)\subset (T^*\Sigma,h_{\nabla})$ -- satisfies the inequality 
\[
\mathrm{Area}(\nabla)\geqslant-2\pi\chi(\Sigma).
\] 
Therefore, we call a spacelike minimal Lagrangian connection with area $-2\pi\chi(\Sigma)$ \textit{area minimising}. We obtain:
\setcounter{section}{7}
\setcounter{thm}{8}
\begin{thm}
Let $\Sigma$ be a compact oriented surface with $\chi(\Sigma)<0$. Then we have:
\begin{itemize}
\item[(i)] there exists a one-to-one correspondence between area minimising Lagrangian connections on $T\Sigma$ and pairs $([g],\beta)$ consisting of a conformal structure $[g]$ and a closed $1$-form $\beta$ on $\Sigma$;
\item[(ii)] there exists a one-to-one correspondence between non-area minimising minimal Lagrang\-ian connections on $T\Sigma$ and pairs $([g],C)$ consisting of a conformal structure $[g]$ and a non-trivial cubic differential $C$ on $\Sigma$ that satisfies $\ov{\partial} C=\left(\beta-\i\star_g \beta\right)\otimes C$ for some closed $1$-form $\beta$. 
\end{itemize}
\end{thm}
\setcounter{section}{1}
Using the classification of properly convex projective structures by Loftin \cite{MR1828223} and Labourie~\cite{MR2402597}, it follows that every properly convex projective structure on $\Sigma$ arises from a unique spacelike minimal Lagrangian connection, paralleling the result of Hildebrand.

\subsection{Related work} After the first version of this article appeared on the arXiv, Daniel Fox informed the author about his interesting paper~\cite{MR3137456}, which contains some closely related results. Here we briefly compare our results which were arrived at independently. In a previous eprint~\cite{arXiv:0909.1897} (see also~\cite{MR3738964}), Fox introduced the notion of an AH (affine hypersurface) structure which is a pair comprising a projective structure and a conformal structure satisfying a compatibility condition which is automatic in two dimensions. He then proceeds to postulate Einstein equations for AH structures, which are motivated by Calderbank's work on Einstein--Weyl structures on surfaces~\cite{MR1656822}. Subsequently in~\cite{MR3137456}, Fox classifies the Einstein AH structures on closed oriented surfaces and in particular, observes that in the case of negative Euler-characteristic they precisely correspond to the properly convex projective structures. On the $2$-sphere $S^2$ he recovers the Einstein--Weyl structures of Calderbank. On $S^2$, the Einstein AH structures and the minimal Lagrangian connections ``overlap'' in the space of metrics of constant positive Gauss--curvature. In the case of negative Euler characteristic, the minimal Lagrangian connections are however~\textit{strictly} more general than the Einstein AH structures. Indeed, in this case, the projective structures arising from minimal Lagrangian connections provide a new and previously unstudied class of projective structures. 

This new class of (possibly) curved projective structures may be thought of as a generalization of the notion of a (flat) properly convex projective structure. One would expect that this class exhibits interesting properties, similar to those of properly convex projective structures. As a first result in this direction, it is shown in~\cite{MR3968880}, that the geodesics of a minimal Lagrangian connection naturally give rise to a flow admitting a dominated splitting (a certain weakening of the notion of an Anosov flow). In particular, this flow provides a generalization of the geodesic flow induced by the Hilbert metric on the quotient surface of a divisible convex set. 

Furthermore, in joint work in progress by the author and A.~Cap~\cite{CapMettler19}, the notion of a minimal Lagrangian connection is extended to all so-called $|1|$-graded parabolic geometries. This is a class of geometric structures which, besides projective geometry,   includes (but is not restricted to) conformal geometry, (almost) Grassmannian geometry and (almost) quaternionic geometry. 

\subsection*{Acknowledgements} The author is grateful to Andreas Cap, Daniel Fox, Gabriel Paternain, Nigel Hitchin, Norbert Hunger\-b\"uhler, Tobias Weth and Luca Galimberti for helpful conversations or correspondence. The author also would like to thank the anonymous referees for several helpful suggestions.

\section{Preliminaries}

Throughout the article $\Sigma$ will denote an oriented smooth $2$-manifold with\-out boundary. All manifolds and maps are assumed to be smooth and we adhere to the convention of summing over repeated indices. 

\subsection{The coframe bundle}

We denote by $\upsilon : F \to \Sigma$ the bundle of orientation preserving coframes whose fibre at $p \in \Sigma$ consists of the linear isomorphisms $f : T_p\Sigma \to \R^2$ that are orientation preserving with respect to the fixed orientation on $\Sigma$ and the standard orientation on $\R^2$. Recall that $\upsilon : F \to \Sigma$ is a principal right $\mathrm{GL}^+(2,\R)$-bundle with right action defined by the rule $R_a(f)=f\cdot a=a^{-1}\circ f$ for all $a \in \mathrm{GL}^+(2,\R)$. The bundle $F$ is equipped with a tautological $\R^2$-valued $1$-form $\omega=(\omega^i)$ defined by $\omega_f=f\circ \upsilon^{\prime}_f$, and this $1$-form satisfies the equivariance property $R_a^*\omega=a^{-1}\omega$. A torsion-free connection $\nabla$ on $T\Sigma$ corresponds to a $\mathfrak{gl}(2,\R$)-valued connection $1$-form $\theta=(\theta^i_j)$ on $F$ satisfying the structure equations
\begin{align}\label{eq:struceqlincon1}
\d \omega&=-\theta\wedge \omega,\\ \label{eq:struceqlincon2}
\d\theta&=-\theta\wedge\theta+\Theta,
\end{align}
where $\Theta$ denotes the curvature $2$-form of $\theta$. The Ricci curvature of $\nabla$ is the (not necessarily symmetric) covariant $2$-tensor field $\mathrm{Ric}(\nabla)$ on $\Sigma$ satisfying
\[
\mathrm{Ric}(\nabla)(X,Y)=\tr \left( Z\mapsto \nabla_Z\nabla_X Y-\nabla_X\nabla_Z Y-\nabla_{[Z,X]}Y\right), \quad Z\in \Gamma(TM),
\]
for all vector fields $X,Y$ on $\Sigma$. Denoting by $\mathrm{Ric}^{\pm}(\nabla)$ the symmetric (respectively, the anti-symmetric) part of the Ricci curvature of $\nabla$, so that $\mathrm{Ric}(\nabla)=\mathrm{Ric}^{+}(\nabla)+\mathrm{Ric}^{-}(\nabla)$, the (projective) Schouten tensor of $\nabla$ is defined as
\[
\mathrm{Schout}(\nabla)=\mathrm{Ric}^+(\nabla)-\frac{1}{3}\mathrm{Ric}^{-}(\nabla). 
\]
Since the components of $\omega$ are a basis for the $\upsilon$-semibasic forms on $F$,\footnote{Recall that a $1$-form $\alpha \in \Omega^1(M)$ is semibasic for the projection $\pi : M \to N$ if $\alpha$ vanishes on vector fields that are tangent to the $\pi$-fibres.} it follows that there exist real-valued functions $S_{ij}$ on $F$ such that 
\[
\upsilon^*\mathrm{Schout}(\nabla)=\omega^tS\omega=S_{ij}\omega^i\otimes \omega^j,
\]
where $S=(S_{ij})$. Note that 
\[
R_a^*S=a^tSa
\] for all $a \in \mathrm{GL}^+(2,\R)$, since $\omega^tS\omega$ is invariant under $R_a$. In terms of the functions $S_{ij}$ the curvature $2$-form $\Theta=(\Theta^i_j)$ can be written as\footnote{For a matrix $S=(S_{ij})$ we denote by $S_{(ij)}$ its symmetric part and by $S_{[ij]}$ its anti-symmetric part, so that $S_{ij}=S_{(ij)}+S_{[ij]}$.}
\begin{equation}\label{eq:curvform}
\Theta^i_j=\left(\delta^i_{[k}S_{l]j}-\delta^i_jS_{[kl]}\right)\omega^k\wedge\omega^l,
\end{equation}
or explicitly
\begin{equation}\label{eq:curv2formex}
\Theta=\begin{pmatrix} 2S_{21}-S_{12} & S_{22}\\ -S_{11} & S_{21}-2S_{12}\end{pmatrix}\omega^1\wedge\omega^2.
\end{equation}

\subsection{The orthonormal coframe bundle}
Recall that if $g$ is a Riemannian metric on the oriented surface $\Sigma$, the Levi-Civita connection $(\varphi^i_j)$ of $g$ is the unique connection on the coframe bundle $\upsilon : F \to \Sigma$ satisfying 
\begin{align*}
\d \omega^i&=-\varphi^i_j\wedge\omega^j,\\
\d g_{ij}&=g_{ik}\varphi^k_j+g_{kj}\varphi^k_i,
\end{align*}
where we write $\upsilon^*g=g_{ij}\omega^i\otimes\omega^j$ for real-valued functions $g_{ij}=g_{ji}$ on $F$. Differentiating these equations implies that there exists a unique function $K_g$, the Gauss curvature of $g$, so that
\[
\d \varphi^i_j+\varphi^i_k\wedge\varphi^k_j=g_{jk}K_g\omega^i\wedge\omega^k. 
\] 
We may reduce $F$ to the $\mathrm{SO}(2)$-subbundle $F_g$ consisting of orientation preserving coframes that are also orthonormal with respect to $g$ , that is, the bundle defined by the equations $g_{ij}=\delta_{ij}$. On $F_g$ the identity $\d g_{ij}= 0$ implies the identities $\varphi^1_1=\varphi^2_2= 0$ as well as $\varphi^1_2+\varphi^2_1= 0$. Therefore, writing $\varphi:=\varphi^2_1$, we obtain the structure equations
\begin{align}\label{eq:struceqriemcon}
\begin{split}
\d \omega_1&=-\omega_2\wedge\varphi,\\
\d\omega_2&=-\varphi\wedge\omega_1,\\
\d \varphi&=-K_g\omega_1\wedge\omega_2,
\end{split}
\end{align}
where $\omega_i=\delta_{ij}\omega^j$. 
Continuing to denote the basepoint projection $F_g \to \Sigma$ by $\upsilon$, the area form $dA_g$ of $g$ satisfies $\upsilon^*dA_g=\omega_1\wedge\omega_2$. Also, note that a complex-valued $1$-form $\alpha$ on $\Sigma$ is a $(1,\! 0)$-form for the complex structure $J$ induced on $\Sigma$ by $g$ and the orientation if and only if $\upsilon^*\alpha$ is a complex multiple of the complex-valued form $\omega=\omega_1+\i\omega_2$. In particular, denoting by $K_{\Sigma}$ the canonical bundle of $\Sigma$ with respect to $J$, a section $A$ of the $\ell$-th tensorial power of $K_{\Sigma}$ satisfies $\upsilon^*A=a\omega^{\ell}$ for some unique complex-valued function $a$ on $F_g$. Denote by $S^3_0(T^*\Sigma)$ the trace-free part of $S^3(T^*\Sigma)$ with respect to $[g]$, where $S^3(T^*\Sigma)$ denotes the third symmetrical power of the cotangent bundle of $\Sigma$. The proof of the following lemma is an elementary computation and thus omitted. 
\begin{lem}\label{lem:cubisymtrcor3}
Suppose $W \in \Gamma\left(S^3_0(T^*\Sigma)\right)$. Then there exists a unique cubic differential $C \in \Gamma(K_{\Sigma}^3)$ so that $\Re(C)=W$. Moreover, writing $\upsilon^*W=w_{ijk}\omega_i\otimes\omega_j\otimes\omega_k$ for unique real-valued functions $w_{ijk}$ on $F_g$, totally symmetric in all indices, the cubic differential satisfies $\upsilon^*C=(w_{111}+\i w_{222})\omega^3$.    
\end{lem}
In complex notation, the structure equations of a cubic differential $C \in \Gamma(K_{\Sigma}^3)$ can be written as follows. Writing $\upsilon^*C=c\omega^3$ for a complex-valued function $c$ on $F_g$, it follows from the $\mathrm{SO}(2)$-equivariance of $c\omega^3$ that there exist complex-valued functions $c^{\prime}$ and $c^{\prime\prime}$ on $F_g$ such that
\[
\d c=c^{\prime}\omega+c^{\prime\prime}\ov{\omega}+3\i c\varphi,
\]
where we write $\ov{\omega}=\omega_1-\i \omega_2$. Note that the Hermitian metric induced by $g$ on $K_{\Sigma}^3$ has Chern connection $\mathrm{D}$ given by 
\[
c \mapsto \d c-3\i c\varphi.
\]
In particular, the $(0,\! 1)$-derivative of $C$ with respect to $\mathrm{D}$ is represented by $c^{\prime\prime}$, that is, $\upsilon^*(\mathrm{D}^{0,1}C)=c^{\prime\prime}\omega^3\otimes \ov{\omega}$. Since $\ov{\partial}=\mathrm{D}^{0,1}$, we obtain
\begin{equation}\label{eq:delbarframe}
\upsilon^*\left(\ov{\partial}C\right)=c^{\prime\prime}\omega^3\otimes\ov{\omega}.
\end{equation}
Also, we record the identity 
\[
\upsilon^*|C|_g^2=|c|^2. 
\]

Moreover, recall that for $u \in C^{\infty}(\Sigma)$ we have the following standard identity for the change of the Gauss curvature of a metric $g$ under conformal rescaling 
\[
K_{\e^{2u}g}=\e^{-2u}\left(K_g-\Delta_g u\right),
\]
where $\Delta_g=-\left(\delta_g \d+\d \delta_g\right)$ is the negative of the Laplace--Beltrami operator with respect to $g$. Also, 
\[
dA_{\e^{2u}g}=\e^{2u}dA_g
\]
for the change of the area form $dA_g$, 
\[
\Delta_{\e^{2u}g}=\e^{-2u}\Delta_g 
\]
for $\Delta_g$ acting on functions and
\[
\delta_{\e^{2u}g}=\e^{-2u}\delta_g
\]
for the co-differential acting on $1$-forms. Finally, the norm of $C$ changes as
\[
|C|^2_{\e^{2u}g}=\e^{-6u}|C|^2_{g}.
\] 

\subsection{The cotangent bundle and induced structures} Recall that we have a $\mathrm{GL}^+(2,\R)$ representation $\varrho$ on $\R_2$ -- the real vector space of row vectors of length two with real entries -- defined by the rule $\varrho(a) \xi=\xi a^{-1}$ for all $\xi \in \R_2$ and $a \in \mathrm{GL}^+(2,\R)$. The cotangent bundle of $\Sigma$    is the vector bundle associated to the coframe bundle $F$ via the representation $\varrho$, that is, the bundle obtained by taking the quotient of $F\times \R_2$ by the $\mathrm{GL}^+(2,\R)$-right action induced by $\varrho$. Consequently, a $1$-form on $\Sigma$ is represented by an $\R_2$-valued function $\xi$ on $F$ which is $\mathrm{GL}^+(2,\R)$-equivariant, that is, $\xi$ satisfies $R_a^*\xi=\xi a$ for all $a \in \mathrm{GL}^+(2,\R)$. 

Using $\theta$ we may define a Riemannian metric $h_{\nabla}$ as well as a symplectic form $\Omega_{\nabla}$ on $T^*\Sigma$ as follows. Let
\[
\pi : F\times \R_2 \to \left(F\times \R_2\right)/\mathrm{GL}^+(2,\R)\simeq T^*\Sigma
\]
denote the quotient projection. Writing 
\[
\psi=\d \xi-\xi\theta-\xi\omega\xi-\omega^tS^t
\] 
or in components $\psi=(\psi_i)$ with
\[
\psi_i=\d \xi_i-\xi_j\theta^j_i-\xi_j\omega^j\xi_i-S_{ji}\omega^j,
\]
we consider the covariant $2$-tensor field $T_{\nabla}=\psi\omega:=\psi_i\otimes \omega^i$. Note that the $\pi$-semibasic $1$-forms on $F\times \R_2$ are given by the components of $\omega$ and $\d \xi-\xi\theta$, or, equivalently, by the components of $\omega$ and $\psi$. Indeed, the components of $\omega$ are semibasic by the definition of $\omega$. Moreover, if $X_v$ for $v \in \mathfrak{gl}{(2,\R)}$ is a fundamental vector field for the coframe bundle $F \to \Sigma$, that is, the vector field associated to the flow $R_{\exp(tv)}$, then
\begin{align*}
(\d \xi-\xi\theta)(X_v)&=-\xi\theta(X_v)+\lim_{t\to 0}\frac{1}{t}\left((R_{\exp(tv)})^*\xi-\xi\right)\\
&=-\xi v+\lim_{t\to 0}\frac{1}{t}\left(\xi\exp(tv)-\xi\right)=-\xi v+\xi v=0,
\end{align*}
where we have used that the connection $\theta$ maps a fundamental vector field $X_v$ to its generator $v \in \mathfrak{gl}(2,\R)$. Since the fundamental vector fields span the vector fields tangent to fibres of $\pi$, it follows that the components of $\d \xi-\xi\theta$ are $\pi$-semibasic.
Moreover, 
\begin{align}
\label{eq:equieta} R_a^*\psi&=\d \xi\, a-\xi a a^{-1}\theta a-\xi a a^{-1}\omega\xi a-\omega^t (a^{-1})^t a^tS^ta=\psi a,\\
\label{eq:equiomega} R_a^*\omega&=a^{-1}\omega, 
\end{align}
for all $a \in \mathrm{GL}^+(2,\R)$, it follows that the $\pi$-semibasic tensor field $T_{\nabla}$ is invariant under the $\mathrm{GL}^+(2,\R)$-right action and hence there exists a unique symmetric covariant $2$-tensor field $h_{\nabla}$ and a unique anti-symmetric covariant $2$-tensor field $\Omega_{\nabla}$ on $T^*\Sigma$ such that
\[
\pi^*\left(h_{\nabla}+\Omega_{\nabla}\right)=T_{\nabla}. 
\]
Using the structure equation~\eqref{eq:struceqlincon1}, we compute
\begin{align*}
\pi^*\Omega_{\nabla}&=\d \xi_i\wedge \omega^i-\xi_j\theta^j_i\wedge\omega^i-\xi_i\xi_j\omega^j\wedge\omega^i-S_{ji}\omega^j\wedge\omega^i\\
&=\d \xi_i\wedge\omega^i+\xi_i\d\omega^i-S_{ij}\omega^i\wedge\omega^j\\
&=\d(\xi_i\omega^i)-S_{[ij]}\omega^i\wedge\omega^j.
\end{align*}
The $1$-form $\xi\omega=\xi_i\omega^i$ on $F\times \R_2$ is $\pi$-semibasic and $R_a$ invariant, hence the $\pi$-pullback of a unique $1$-form $\tau$ on $T^*\Sigma$ which is the tautological $1$-form of $T^*\Sigma$. Recall that the canonical symplectic form on $T^*\Sigma$ is $\Omega_0=\d \tau$, hence $\Omega_{\nabla}$ defines a symplectic structure on $T^*\Sigma$ which is the canonical symplectic structure twisted with the (closed) $2$-form $\frac{1}{3}\mathrm{Ric}^{-}(\nabla)$
\[
\Omega_{\nabla}=\Omega_0+\Upsilon^*\left(\frac{1}{3}\mathrm{Ric}^{-}(\nabla)\right), 
\]
where $\Upsilon : T^*\Sigma \to \Sigma$ denotes the basepoint projection. In particular, denoting by $o : \Sigma \to T^*\Sigma$ the zero $\Upsilon$-section, the definition of the Schouten tensor gives
\[
o^*\Omega_{\nabla}=\frac{1}{3}\mathrm{Ric}^{-}(\nabla). 
\]
This shows:
\begin{ppn}
The zero section of $T^*\Sigma$ is a $\Omega_{\nabla}$-Lagrangian submanifold if and only if $\nabla$ has symmetric Ricci tensor. 
\end{ppn}
Which motivates:
\begin{defn}
A torsion-free connection $\nabla$ on $T\Sigma$ is called~\textit{Lagrangian} if $\mathrm{Ric}^{-}(\nabla)$ vanishes identically. 
\end{defn}

Also, we obtain for the symmetric part
\[
\pi^*h_{\nabla}=\psi\circ \omega:=\psi_1\circ \omega^1+\psi_2\circ \omega^2
\]
where $\circ$ denotes the symmetric tensor product. Since the four $1$-forms $\psi_1,\psi_2,\omega^1,\omega^2$ are linearly independent, it follows that $h_{\nabla}$ is non-degenerate and hence defines a pseudo-Riemannian metric of split signature $(1,1,-1,-1)$ on $T^*\Sigma$. 
\begin{rmk}
The motivation for introducing the pair $(h_{\nabla},\Omega_{\nabla})$ is its projective invariance, i.e., suitably interpreted, the pair $(h_{\nabla},\Omega_{\nabla})$ does only depend on the projective equivalence class of the connection $\nabla$. Moreover, the metric $h_{\nabla}$ is anti-self-dual and Einstein. We refer the reader to~\cite{MR3833818} as well as~\cite{MR3210600} for further details.  
\end{rmk}
From the definition of the Schouten tensor and $h_{\nabla}$ we immediately obtain
\begin{equation}\label{eq:pullbackmetriccon}
o^*h_{\nabla}=-\mathrm{Ric}^+(\nabla).
\end{equation}
Following standard pseudo-Riemannian submanifold theory, we call a tangent vector $v$~\textit{timelike} if $h_{\nabla}(v,v)<0$ and~\textit{spacelike} if $h_{\nabla}(v,v)>0$. Thus~\eqref{eq:pullbackmetriccon} motivates: 
\begin{defn}
A torsion-free connection $\nabla$ on $T\Sigma$ is called~\textit{timelike} if $\mathrm{Ric}^+(\nabla)$ is positive definite and~\textit{spacelike} if $\mathrm{Ric}^+(\nabla)$ is negative definite. 
\end{defn}

\section{Twisted Weyl connections}

We will see that timelike/spacelike minimal Lagrangian connections are \textit{twisted Weyl connections}. In this section we study some properties of this class of connections that we will need later during the classification of spacelike minimal Lagrangian connections. 

Let $[g]$ be a conformal structure on the smooth oriented surface $\Sigma$. By a $[g]$-Weyl connection on $\Sigma$ we mean a torsion-free connection on $T\Sigma$ preserving the conformal structure $[g]$. It follows from Koszul's identity that a $[g]$-Weyl connection can be written in the following form
\[
{}^{(g,\beta)}\nabla={}^g\nabla+g\otimes \beta^{\sharp}-\beta\otimes\mathrm{Id}-\mathrm{Id}\otimes \beta,
\]
where $g \in [g]$, $\beta \in \Omega^1(\Sigma)$ is a $1$-form and $\beta^{\sharp}$ denotes the $g$-dual vector field to $\beta$. We will use the notation ${}^{[g]}\nabla$ to denote a general $[g]$-Weyl connection.  
\begin{defn}\label{defn:twistconf}
A~\textit{twisted Weyl connection} $\nabla$ on $(\Sigma,[g])$ is a connection on the tangent bundle of $\Sigma$ which can be written as $\nabla={}^{[g]}\nabla+\alpha$ for some $[g]$-Weyl connection ${}^{[g]}\nabla$ and some $1$-form $\alpha$ with values in $\mathrm{End}(T\Sigma)$ satisfying the following properties:
\begin{itemize}
\item[(i)] $\alpha(X)$ is trace-free and $[g]$-symmetric for all $X \in \Gamma(T\Sigma)$;  
\item[(ii)] $\alpha(X)Y=\alpha(Y)X$ for all $X,Y \in \Gamma(T\Sigma)$.
\end{itemize}\end{defn}
Note that if $\alpha$ satisfies the above properties, then ${}^{[g]}\nabla+\alpha$ is torsion-free. Moreover, the covariant $3$-tensor obtained by lowering the upper index of $\alpha$ with a metric $g \in [g]$ gives a section of $\Gamma(S^3_0(T^*\Sigma))$. Conversely, every $\mathrm{End}(T\Sigma)$-valued $1$-form on $\Sigma$ satisfying the above properties arises in this way. In other words, fixing a Riemannian metric $g \in [g]$ allows to identify the twist term $\alpha$ with a cubic differential.  

Fixing a metric $g\in [g]$, the connection form $\theta=(\theta^i_j)$ of a twisted Weyl connection is given by
\[
\theta^i_j=\varphi^i_j+\left(b_kg^{ki}g_{jl}-\delta^i_jb_l-\delta^i_lb_j+a^i_{jl}\right)\omega^l,
\]
where the map $(g_{ij}) : F \to S^2(\R_2)$ represents the metric $g$, the map $(b_i) : F \to \R_2$ represents the $1$-form $\beta$ and the map $(a^i_{jk}) : F \to \R^2\otimes S^2(\R_2)$ represents the $1$-form $\alpha$. Moreover, $(\varphi^i_j)$ denote the Levi-Civita connection forms of $g$. Reducing to the bundle $F_g$ of $g$-orthonormal orientation preserving coframes, the connection form becomes
\[
\theta=\begin{pmatrix} -\beta & \star_g\beta-\varphi \\ \varphi-\star_g\beta & -\beta\end{pmatrix}+\begin{pmatrix} a^1_{11}\omega_1+a^1_{12}\omega_2 & a^1_{12}\omega^1+a^1_{22}\omega_2 \\ a^2_{11}\omega_1+a^2_{12}\omega_2 & a^2_{12}\omega_1+a^2_{22}\omega_2\end{pmatrix}, 
\]
where we use the identity $\upsilon^*\left(\star_g\beta\right)=-b_2\omega_1+b_1\omega_2$. By definition, on $F_g$ the functions $a^i_{jk}$ satisfy the identities 
\[
a^{i}_{jk}=a^i_{kj}, \qquad a^{k}_{kj}=0, \qquad \delta_{ki}a^{k}_{jl}=\delta_{kj}a^k_{il}. 
\]
Thus, writing $c_1=a^1_{11}$ and $c_2=a^2_{22}$, we obtain
\[
\theta=\begin{pmatrix} -\beta & \star_g \beta-\varphi \\ \varphi-\star_g\beta & -\beta\end{pmatrix}+\begin{pmatrix} c_1\omega_1-c_2\omega_2 & -c_2\omega_1-c_1\omega_2\\ -c_2\omega_1-c_1\omega_2 & -c_1\omega_1+c_2 \omega_2\end{pmatrix}.
\]
In order to compute the curvature form of $\theta$ we first recall that we write $\upsilon^*\beta=b_i\omega_i$ and since $b_i\omega_i$ is $\mathrm{SO}(2)$-invariant, it follows that there exist unique real-valued functions $b_{ij}$ on $F_g$ such that
\begin{align*}
\d b_1&=b_{11}\omega_1+b_{12}\omega_2+b_2\varphi,\\
\d b_2&=b_{21}\omega_1+b_{22}\omega_2-b_1\varphi.
\end{align*}
Recall also that the area form of $g$ satisfies $\upsilon^*dA_g=\omega_1\wedge\omega_2$ and since $\star_g 1=dA_g$, we get 
\[
\upsilon^*\delta_g \beta=-(b_{11}+b_{22}),
\]
as well as
\[
\upsilon^*\left(\d\!\star_g\!\beta\right)=(b_{11}+b_{22})\omega_1\wedge\omega_2.
\]
Since $c_1+\i c_2$ represents a cubic differential on $\Sigma$, there exist unique real-valued functions $c_{ij}$ on $F_g$ such that
\begin{align*}
\d c_1&=c_{11}\omega_1+c_{12}\omega_2-3c_2\varphi,\\
\d c_2&=c_{21}\omega_1+c_{22}\omega_2+3c_1\varphi.
\end{align*}
Consequently, a straightforward calculation shows that the curvature form $\Theta=\d\theta+\theta\wedge\theta$ satisfies
\begin{multline}\label{eq:curv2formtwistconfcon}
\Theta=\begin{pmatrix} -\d \beta & K_gdA_g+\d \star_g\beta-\frac{1}{2}|\alpha|^2_g\omega_1\wedge\omega_2\\ -K_gd A_g- \d \star_g\beta+\frac{1}{2}|\alpha|^2_g\omega_1\wedge\omega_2& -\d\beta  \end{pmatrix}\\ + \begin{pmatrix} 2(b_1c_2+b_2c_1)-(c_{12}+c_{21}) & 2(b_1c_1-b_2c_2)+\left(c_{22}-c_{11}\right)\\ 2(b_1c_1-b_2c_2)+\left(c_{22}-c_{11}\right) & 2(-b_1c_2-b_2c_1)+(c_{12}+c_{21})\end{pmatrix} \omega_1\wedge\omega_2,
\end{multline}
where we use the identity $\upsilon^*|\alpha|^2_g=4\left((c_1)^2+(c_2)^2\right)$. 

\subsection{A characterisation of twisted Weyl connections}

We obtain a natural differential operator $\mathrm{D}_{[g]}$ acting on the space $\mathfrak{A}(\Sigma)$ of torsion-free connections on $T\Sigma$
\[
\mathrm{D}_{[g]} : \mathfrak{A}(\Sigma) \to \Omega^2(\Sigma), \quad \nabla \mapsto \tr_g \mathrm{Ric}(\nabla)dA_g.
\]
Note that this operator does indeed only depend on the conformal equivalence class of $g$. A twisted Weyl connection $\nabla$ on $(\Sigma,[g])$ can be characterised by minimising the integral of $\mathrm{D}_{[g]}$ among its projective equivalence class $\mathfrak{p}(\nabla)$. 
\begin{ppn}\label{ppn:inftwistedconf}
Suppose $\nabla^{\prime}={}^{[g]}\nabla+\alpha$ is a twisted Weyl connection on the compact Riemann surface $(\Sigma,[g])$. Then
\[
\inf_{\nabla \in \mathfrak{p}(\nabla^{\prime})}\int_{\Sigma}\mathrm{D}_{[g]}(\nabla)=4\pi\chi(\Sigma)-\Vert\alpha\Vert^2_g
\]
and $4\pi\chi(\Sigma)-\Vert\alpha\Vert^2_g$ is attained precisely on $\nabla^{\prime}$. 
\end{ppn}
\begin{rmk}
Note that 
\[
\Vert \alpha\Vert^2_{g}=\int_{\Sigma}|\alpha|^2_g\,dA_g
\] does only depend on the conformal equivalence class of $g$. 
\end{rmk}
\begin{proof}[Proof of \cref{ppn:inftwistedconf}]
Write $\nabla^{\prime}={}^{(g,\beta)}\nabla+\alpha$ for some Riemannian metric $g \in [g]$, some $1$-form $\beta$ and some $\mathrm{End}(T\Sigma)$-valued $1$-form $\alpha$ on $\Sigma$ satisfying the properties of \cref{defn:twistconf}. From~\eqref{eq:curv2formex} and the definition of the Schouten tensor it follows that
\[
\upsilon^*\left(\tr_g\mathrm{Ric}(\nabla^{\prime})dA_g\right)=\Theta^1_2-\Theta^2_1,
\]
where $\Theta=(\Theta^i_j)$ denotes the curvature form of $\nabla^{\prime}$ pulled-back to $F_g$. Thus, equation \eqref{eq:curv2formex} gives
\begin{equation}\label{eq:pregbtwistweyl}
\tr_g\mathrm{Ric}(\nabla^{\prime})dA_g=2K_g+2\d\star_g\beta-|\alpha|^2_g dA_g
\end{equation}
and hence
\begin{equation}\label{eq:gbtwistedweyl}
\int_\Sigma\tr_g\mathrm{Ric}(\nabla^{\prime})dA_g=4\pi\chi(\Sigma)-\Vert\alpha\Vert^2_g
\end{equation}
by the Stokes and the Gauss--Bonnet theorem. 

It is a classical result due to Weyl~\cite{zbMATH02603060} that two torsion-free connections $\nabla^{1},\nabla^2$ on $T\Sigma$ are projectively equivalent if and only if there exists a $1$-form $\gamma$ on $\Sigma$ such that $\nabla^{1}-\nabla^2=\gamma\otimes \mathrm{Id}+\mathrm{Id}\otimes \gamma$. It follows that the connections in the projective equivalence class of $\nabla^{\prime}$ can be written as
\[
\nabla=\nabla^{\prime}+\gamma\otimes \mathrm{Id}+\mathrm{Id}\otimes \gamma
\]
with $\gamma \in \Omega^1(\Sigma)$. A simple computation gives
\begin{equation}\label{eq:ricprojchange}
\mathrm{Ric}(\nabla)=\mathrm{Ric}(\nabla^{\prime})+\gamma^2-\mathrm{Sym}\,\nabla^{\prime}\gamma+3\,\d\gamma,
\end{equation}
where $\mathrm{Sym} : \Gamma(T^*\Sigma\otimes T^*\Sigma) \to \Gamma(S^2(T^*\Sigma))$ denotes the natural projection. We compute
\begin{align*}
\tr_g \mathrm{Sym}\nabla^{\prime}\gamma\, dA_g&=\tr_g \mathrm{Sym}\left({}^g\nabla+g \otimes \beta^{\sharp}-\beta\otimes\mathrm{Id}-\mathrm{Id}\otimes \beta+\alpha\right)\gamma\,dA_g\\
&=\d\star_g\gamma+\left(2\gamma(\beta^{\sharp})-\gamma(\beta^{\sharp})-\gamma(\beta^{\sharp})\right)dA_g\\
&=\d\star_g\gamma,
\end{align*}
where we used that $\alpha(X)$ is trace-free and $[g]$-symmetric for all $X \in\Gamma(T\Sigma)$. Since the last summand of the right hand side of~\eqref{eq:ricprojchange} is anti-symmetric, we obtain
\begin{align*}
\int_\Sigma\tr_g \mathrm{Ric}(\nabla)dA_g&=\int_\Sigma\tr_g\mathrm{Ric}(\nabla)+\int_{\Sigma}\tr_g\gamma^2dA_g-\int_{\Sigma}\tr_g\mathrm{Sym}\nabla^{\prime}\gamma\, dA_g\\
&=4\pi\chi(\Sigma)-\Vert\alpha\Vert^2_g+\Vert \gamma\Vert^2_g-\int_{\Sigma}\d \star_g\gamma,
\end{align*} 
thus the claim follows from the Stokes theorem.  
\end{proof}

In~\cite{arXiv:1510.01043} the following result is shown, albeit phrased in different language:
\begin{ppn}\label{ppn:strucconsurf}
Let $(\Sigma,[g])$ be a Riemann surface. Then every torsion-free connection on $T\Sigma$ is projectively equivalent to a unique twisted $[g]$-Weyl connection.
\end{ppn}
Let $\mathfrak{P}(\Sigma)$ denote the space of projective structures on $\Sigma$. Using \cref{ppn:inftwistedconf} and \cref{ppn:strucconsurf} we immediately obtain: 
\begin{thm}\label{thm:minmax}
Let $(\Sigma,[g])$ be a compact Riemann surface. Then 
\[
\sup_{\mathfrak{p} \in \mathfrak{P}(\Sigma)}\inf_{\nabla \in \mathfrak{p}} \int_{\Sigma}\tr_g \mathrm{Ric}(\nabla)dA_g=4\pi\chi(\Sigma).
\]
\end{thm}
\begin{rmk}\label{rmk:alignedrep}
A twisted Weyl connection $\nabla$ on $(\Sigma,[g])$ defines an AH structure $(\mathfrak{p}(\nabla),[g])$ in the sense of~\cite{arXiv:0909.1897,MR3137456,MR3738964}, where $\mathfrak{p}(\nabla)$ denotes the projective equivalence class arising from $\nabla$. Moreover, the twisted Weyl connection agrees with the~\textit{aligned representative} of the associated AH structure $(\mathfrak{p}(\nabla),[g])$. In particular, the equations~\eqref{eq:pregbtwistweyl} and~\eqref{eq:gbtwistedweyl} have counterparts in the equations (5.8) and (7.12) of~\cite{MR3137456}. Also, the \cref{ppn:strucconsurf} corresponds to the existence of a unique aligned representative for an AH structure from~\cite{MR3137456}. 
\end{rmk}

\section{Submanifold theory of Lagrangian connections}

We now restrict attention to torsion-free connections on $T\Sigma$ having symmetric Ricci tensor, so that the zero-section $o : \Sigma \to T^*\Sigma$ is a Lagrangian submanifold. If furthermore $\nabla$ is timelike/spacelike, we obtain an induced metric $g=\mp o^*h_{\nabla}=\pm \mathrm{Ric}(\nabla)$ and the immersion $o : \Sigma \to T^*\Sigma$ has a well-defined normal bundle and second fundamental form. In particular, we want to compute when $o : \Sigma \to T^*\Sigma$ is~\textit{minimal}, that is, the trace with respect to $g$ of its second fundamental form vanishes identically. 

\subsection{Algebraic preliminaries}

Before we delve into the computations, we briefly review the relevant algebraic structure of the theory of oriented surfaces in an oriented Riemannian -- and oriented split-signature Riemannian four manifold $(M,g)$. We refer the reader to~\cite{MR679067} and~\cite{MR692106} for additional details.

First, let $X :\Sigma \to (M,g)$ be an immersion of an oriented surface $\Sigma$ into an oriented Riemannian $4$-manifold. The bundle of orientation compatible $g$-orthonormal coframes of $(M,g)$ is an $\mathrm{SO}(4)$-bundle $\pi : F^+_g \to M$. The Grassmannian $G^+_2(\R^4)$ of oriented $2$-planes in $\R^4$ is a homogeneous space for the natural action of $\mathrm{SO}(4)$ and the stabiliser subgroup is $\mathrm{SO}(2)\times \mathrm{SO}(2)$. Consequently, the pullback bundle $X^*F^+_g \to \Sigma$  admits a reduction $F_X\subset X^*F^+_g$ with structure group $\mathrm{SO}(2)\times \mathrm{SO}(2)$, where the fibre of $F_X$ at $p \in \Sigma$ consists of those coframes mapping the oriented tangent plane to $\Sigma$ at $X(p)$ to some fixed oriented $2$-plane in $\R^4$, while preserving the orientation.  

The second fundamental form of $X$ is a quadratic form on $T\Sigma$ with values in the rank two normal bundle of $X$. Therefore, it is represented by a map $F_X \to S^2(\R_2)\otimes \R^2$ which is equivariant with respect to some suitable representation of $\mathrm{SO}(2)\times \mathrm{SO}(2)$ on $S^2(\R_2)\times \R^2$. The relevant representation is defined by the rule
\[
\varrho(r_{\alpha},r_{\beta})(A)(x,y)=r_{-\beta}A(r_{\alpha}x,r_{\alpha}y), \quad x,y \in \R^2,
\]
where $A \in S^2(\R_2)\otimes \R^2$ is a symmetric bilinear form on $\R^2$ with values in $\R^2$ and $r_{\alpha},r_{\beta}$ denote counter-clockwise rotations in $\R^2$ by the angle $\alpha,\beta$, respectively. As usual, we decompose the $\mathrm{SO}(2)\times \mathrm{SO}(2)$-module $S^2(\R_2)\otimes \R^2$ into irreducible pieces. This yields 
\[\varrho=\varsigma_{0,-1}\oplus\varsigma_{2,1}\oplus\varsigma_{2,-1},
\]
where for $(n,m) \in \mathbb{Z}^2$ the complex one-dimensional $\mathrm{SO}(2)\times \mathrm{SO}(2)$-repre\-sen\-ta\-tion $\varsigma_{n,m}$ is defined by the rule 
\[
\varsigma_{n,m}(r_{\alpha},r_{\beta})=\e^{\i (n \alpha+m\beta)}. 
\]
Explicitly, the relevant projections $S^2(\R_2)\otimes \R^2 \to \C$ are
\begin{align}
\label{eq:defmeanvec} p_{0,-1}(A)&=\frac{1}{2}\left(A^1_{11}+A^1_{22}\right)+\frac{\i}{2}\left(A^2_{11}+A^2_{22}\right),\\
\label{eq:def10form}p_{2,-1}(A)&=\frac{1}{4}\left(A^1_{11}-A^1_{22}+2A^2_{12}\right)+\frac{\i}{4}\left(A^2_{11}-A^2_{22}-2A^1_{12}\right),\\
\label{eq:defcubicdiff}p_{2,1}(A)&=\frac{1}{4}\left(A^1_{22}-A^1_{11}+2A^2_{12}\right)+\frac{\i}{4}\left(A^2_{11}-A^2_{22}+2A^1_{12}\right)
\end{align}
and where we write $A(e_i,e_j)=A^k_{ij}e_k$ with respect to the standard basis $(e_1,e_2)$ of $\R^2$. 

The canonical bundle $K_{\Sigma}$ of $\Sigma$ with respect to the complex structure induced by the  metric $X^*g$ and orientation is the bundle associated to the representation $\varsigma_{1,0}$. Moreover, the conormal bundle, thought of as a complex line bundle, is the bundle $N^*_X$ associated to the representation $\varsigma_{0,1}$. Consequently, the second fundamental form of $X$ defines a section $H$ of the normal bundle which is the mean curvature vector of $X$, as well as a quadratic differential $Q_+$with values in the conormal bundle and a quadratic differential $Q_-$ with values in the normal bundle. Consequently, we obtain a quartic differential $Q_+Q_-$ on $\Sigma$ which turns out to be holomorphic, provided $X$ is minimal and $g$ has constant sectional curvature, see~\cite{MR692106}.

If we instead consider a split-signature oriented Riemannian $4$-manifold $(M,g)$, the bundle of orientation compatible $g$-orthonormal coframes of $(M,g)$ is an $\mathrm{SO}(2,2)$-bundle $\pi : F^+_g \to M$. Here, as usual, we take $\mathrm{SO}(2,2)$ to be the subgroup of $\mathrm{SL}(4,\R)$ stabilising the quadratic form
\[
q(x)=(x_1)^2+(x_2)^2-(y_1)^2-(y_2)^2, 
\]
where $(x,y) \in \R^{2,2}$. Now the action of $\mathrm{SO}(2,2)$ on the Grassmannian $G^+_2(\R^{2,2})$ of oriented $2$-planes in $\R^{2,2}$ is not transitive, it is however transitive on the open submanifolds of oriented timelike/spacelike $2$-planes. In both cases, the stabiliser subgroup is $\mathrm{SO}(2)\times\mathrm{SO}(2)$ as well. Therefore, the submanifold theory of a timelike/spacelike oriented surface in an oriented split-signature Riemannian four manifold is entirely analogous to the Riemannian case. In particular, we also encounter the mean curvature vector $H$ and the quadratic differentials $Q_{\pm}$. 

\subsection{The mean curvature form}

Knowing what to expect, we now carry out the submanifold theory of timelike/spacelike Lagrangian connections. 
Note however, that in addition to the split-signature metric $h_{\nabla}$, we also have a symplectic form $\Omega_{\nabla}$. The symplectic form allows to identify the conormal bundle to a Lagrangian spacelike/timelike immersion with the cotangent bundle of $\Sigma$. In particular we may think of the mean curvature vector $H$ as a $1$-form, the conormal-bundle valued quadratic differential $Q_+$ as a cubic differential and the normal-bundle valued quadratic differential $Q_-$ as a $(1,\! 0)$-form.

The product $P:=F\times \R_2$ is a principal right $\mathrm{GL}^+(2,\R)$-bundle over $T^*\Sigma$, where the $\mathrm{GL}^+(2,\R)$-right action is given by $(f,\xi)\cdot a=(a^{-1}\circ f,\xi a)$ for all $a \in \mathrm{GL}^+(2,\R)$ and $(f,\xi) \in P$. We define two $\R^2$-valued $1$-forms on $P$
\[
\rho:=\frac{1}{2}\left(\psi^t+\omega\right)\quad \text{and}\quad \zeta:=\frac{1}{2}\left(\psi^t-\omega\right),
\]
so that the metric $h_{\nabla}$ satisfies
\[
\pi^*h_{\nabla}=\rho^t\rho-\zeta^t\zeta=(\rho^1)^2+(\rho^2)^2-(\zeta^1)^2-(\zeta^2)^2.
\]
From the equivariance properties~\eqref{eq:equieta} and~\eqref{eq:equiomega} of $\psi$ and $\omega$, we compute
\begin{equation}\label{eq:cotangentbundelasso}
R_a^*\begin{pmatrix} \rho \\ \zeta \end{pmatrix}=\frac{1}{2}\begin{pmatrix} a^t+a^{-1} & a^t-a^{-1} \\ a^t-a^{-1} & a^t+a^{-1}\end{pmatrix}\begin{pmatrix} \rho \\ \zeta \end{pmatrix}.
\end{equation}
\begin{rmk}
The reader may easily verify that the representation $\mathrm{GL}^+(2,\R) \to \mathrm{GL}(4,\R)$ defined by~\eqref{eq:cotangentbundelasso} embeds $\mathrm{GL}^+(2,\R)$ as a subgroup of $\mathrm{SO}(2,2)$. 
\end{rmk}
Recall that the Lie algebra of the split-orthogonal group $\mathrm{O}(2,2)$ consists of matrices of the form
\[
\begin{pmatrix} \mu & \nu \\ \nu^t & \vartheta \end{pmatrix}
\]
where $\mu$ and $\vartheta$ are skew-symmetric. Consequently, there exist unique $\mathfrak{o}(2)$-valued $1$-forms $\mu,\vartheta$ on $F\times \R_2$ and a unique $\mathfrak{gl}(2,\R)$-valued $1$-form $\nu$ on $F\times \R_2$ such that
\begin{equation}\label{eq:struceqriemmetri}
\d \begin{pmatrix} \rho \\ \zeta\end{pmatrix}=-\begin{pmatrix} \mu & \nu \\ \nu^t & \vartheta \end{pmatrix}\wedge \begin{pmatrix} \rho \\ \zeta\end{pmatrix}.
\end{equation}
In order to compute these connection forms we first remark that since the function $S=(S_{ij})$ represents the (symmetric) Ricci tensor of $\nabla$, there must exist unique real-valued functions $S_{ijk}=S_{jik}$ on $F$ so that
\[
\d S_{ij}=S_{ijk}\omega^k+S_{ik}\theta^k_j+S_{kj}\theta^k_i. 
\]
Clearly, the function $(S_{ijk}) : F \to S^2(\R_2)\otimes \R_2$ represents $\nabla\, \mathrm{Ric}(\nabla)$ with $k$ being the derivative index. 
\begin{lem}\label{lem:conforms}
We have
\begin{align*}
\mu_{ij}&=-\xi_{[i}\omega_{j]}+\theta_{[ij]}-S_{k[ij]}\omega^k,\\
\nu_{ij}&=-\xi_k\omega_k\delta_{ij}-\xi_{(i}\omega_{j)}-\theta_{(ij)}+S_{k[ij]}\omega_k,\\
\vartheta_{ij}&=-\xi_{[i}\omega_{j]}+\theta_{[ij]}+S_{k[ij]}\omega^k\omega_k,
\end{align*}
where we write $\omega_i=\delta_{ij}\omega^j$ and $\theta_{ij}=\delta_{ik}\theta^k_j$. 
\end{lem}
\begin{proof}
Since the connection forms are unique, the proof amounts to plugging the above formulae into the structure equations~\eqref{eq:struceqriemmetri} and verify that they are satisfied. This is tedious, but an elementary computation and hence is omitted. 
\end{proof}

Recall that the components of $\psi$ and $\omega$ -- and hence equivalently the components of $\rho$ and $\zeta$ -- span the $1$-forms on $P$ that are semi-basic for the projection $\pi : P \to T^*\Sigma$. In particular, if $\epsilon$ is a $1$-form on $T^*\Sigma$, then there exists a unique map $(e_1,e_2) : P \to \R_{2,2}$ so that $\pi^*\epsilon=e_1\rho+e_2\zeta$. Since $\pi^*\epsilon$ is invariant under the $\mathrm{GL}^+(2,\R)$ right action, the function $(e_1,e_2)$ satisfies the equivariance property determined by~\eqref{eq:cotangentbundelasso}. Phrased differently, the cotangent bundle of $T^*\Sigma$ is the bundle associated to $\pi : P\to T^*\Sigma$ via the representation $\varrho : \mathrm{GL}^+(2,\R) \to \R_{2,2}$ defined by the rule
\begin{equation}\label{eq:pullbackbundledoublecotangent}
\varrho(a)\begin{pmatrix} \xi_1 & \xi_2\end{pmatrix}=\begin{pmatrix} \xi_1 & \xi_2\end{pmatrix}\frac{1}{2}\begin{pmatrix} a^t+a^{-1} & a^t-a^{-1} \\ a^t-a^{-1} & a^t+a^{-1}\end{pmatrix}
\end{equation}
for all $a \in \mathrm{GL}^+(2,\R)$ and $(\xi_1,\xi_2)$ in $\R_{2,2}$. 

We will next use this fact to exhibit the conormal bundle of the immersion $o : \Sigma \to T^*\Sigma$ as an associated bundle to a natural reduction of the pullback bundle $o^*P$. Note that by construction, the pullback bundle $o^*P \to \Sigma$ is just the frame bundle $\upsilon : F \to \Sigma$ and that on $o^*P\simeq F$ we have $\psi^t=-S\omega$, thus
\[
\rho=\frac{1}{2}\left(\mathrm{I}_2-S\right)\omega \quad \text{and}\quad \zeta=-\frac{1}{2}\left(\mathrm{I}_2+S\right)\omega. 
\] If we assume that $\nabla$ is timelike/spacelike, then the Ricci tensor of $\nabla$ is positive/negative definite and hence the equations $S=\pm \mathrm{I}_2$ define a reduction $F_{\nabla} \to \Sigma$ with structure group $\mathrm{SO}(2)$ whose basepoint projection we continue to denote by $\upsilon$. Note that by construction, $F_{\nabla} \to \Sigma$ is the bundle of orientation preserving orthonormal coframes of the induced metric $g=\pm \mathrm{Ric}(\nabla)$. In particular, from~\eqref{eq:pullbackbundledoublecotangent} we see that the pullback bundle $o^*(T^*(T^*\Sigma))$ is the bundle associated to $\upsilon : F_{\nabla} \to \Sigma$ via the $\mathrm{SO}(2)$-representation on $\R_{2,2}$ defined by the rule
\begin{equation}\label{eq:pullbackcotangent}
\varrho(a)\begin{pmatrix} \xi_1 & \xi_2\end{pmatrix}=\begin{pmatrix} \xi_1 a^t & \xi_2 a^t\end{pmatrix}
\end{equation}
for all $a \in \mathrm{SO}(2)$ and $(\xi_1,\xi_2)$ in $\R_{2,2}$. Furthermore, on $F_{\nabla}$ we obtain $\psi=\mp \omega^t$ and hence 
\begin{equation}\label{eq:frameadapttimelike}
\begin{pmatrix} \rho \\ \zeta \end{pmatrix}=\begin{pmatrix} 0\\ -\omega \end{pmatrix}
\end{equation}
in the timelike case and  
\begin{equation}\label{eq:frameadaptspacelike}
\begin{pmatrix} \rho \\ \zeta \end{pmatrix}=\begin{pmatrix} \omega\\ 0 \end{pmatrix}
\end{equation}
in the spacelike case. Recall that if $\alpha$ is a $1$-form on $\Sigma$ then there exists a unique $\R_2$-valued function $a$ on $F$ -- and hence on $F_{\nabla}$ as well -- so that $\upsilon^*\alpha=a\omega$. It follows as before that $T^*\Sigma$ is the bundle associated to $F_{\nabla}$ via the $\mathrm{SO}(2)$-representation defined by the rule
\begin{equation}\label{eq:repcotangentbundle}
\varrho(a)(\xi)=\xi a^t
\end{equation}  
for all $a \in \mathrm{SO}(2)$ and $\xi \in \R_2$. Using~\eqref{eq:pullbackcotangent},~\eqref{eq:frameadapttimelike} and~\eqref{eq:frameadaptspacelike}, we see that the conormal bundle 
\[
N^*_o:=o^*(T^*(T^*\Sigma))/T^*\Sigma
\] 
of $o$ is (isomorphic to) the bundle associated to $F_{\nabla}$ via the representation~\eqref{eq:repcotangentbundle} as well. We thus have an isomorphism $N^*_o\simeq T^*\Sigma$ between the conormal bundle of the immersion $o : \Sigma \to T^*\Sigma$ and the cotangent bundle $T^*\Sigma$. Of course, the metric $g$ on $\Sigma$ provides an isomorphism $T^*\Sigma\simeq T\Sigma$ and hence $N_o\simeq T^*\Sigma$, where $N_o$ denotes the normal bundle of $o$. The second fundamental form of $o$ is a quadratic form on $T\Sigma$ with values in the normal bundle, thus here naturally a section of $S^2(T^*\Sigma)\otimes T^*\Sigma$. 

\begin{lem}
Let $o : \Sigma \to T^*\Sigma$ be timelike/spacelike. Then the second fundamental form $A$ of $o$ is represented by the functions $A_{ijk}=A_{ikj}$, where 
\begin{equation}\label{eq:2ndfundform}
A_{ijk}=\mp\frac{1}{2}\left(S_{kji}-S_{ijk}-S_{ikj}\right).
\end{equation}
\end{lem} 
\begin{proof}
We will only treat the spacelike case, the timelike case is entirely analogous up to some sign changes. In our frame adaption on $F_{\nabla}$ we have $\zeta=0$ and $\rho=\omega$. Consequently, 
\[
0=\d \zeta=-\nu^t\wedge\rho-\vartheta\wedge\zeta=-\nu^t\wedge\omega
\]
or in components
\[
0=\nu_{ji}\wedge\omega_j.
\]
Cartan's lemma implies that there exist unique real-valued functions $A_{ijk}=A_{ikj}$ on $F_{\nabla}$ so that
\[
\nu_{ji}=A_{ijk}\omega_k
\]
and by standard submanifold theory the functions $A_{ijk}$ represent the second fundamental form of $o$. In order to compute the functions $A_{ijk}$, we use that in our frame adaption $S_{ij}=-\delta_{ij}$ and hence
\[
0=\d S_{ij}=S_{ijk}\omega_k-\delta_{ik}\theta^k_j-\delta_{kj}\theta^k_i=S_{ijk}\omega_k-2\theta_{(ij)}.
\]
Since $\xi=0$ we thus get from \cref{lem:conforms}
\[
\nu_{ji}=-\theta_{(ji)}+S_{k[ji]}\omega_k=\left(-\frac{1}{2}S_{ijk}+\frac{1}{2}S_{kji}-\frac{1}{2}S_{ikj}\right)\omega_k,
\]
and the claim follows.   
\end{proof}
Denoting by $S^{ij}$ the functions on $F$ representing the inverse of the Ricci curvature of $\nabla$, so that $S^{ij}S_{jk}=\delta^i_k$, we thus have:
\begin{thm}\label{thm:minchar}
A timelike/spacelike Lagrangian connection $\nabla$ is minimal if and only if
\begin{equation}\label{eq:minpde}
S^{ij}\left(2S_{kij}-S_{ijk}\right)=0. 
\end{equation}
\end{thm}
\begin{proof}
By standard submanifold theory, the immersion $o : \Sigma \to T^*\Sigma$ is minimal if and only if the trace of the second fundamental form with respect to the induced metric $g=\pm \mathrm{Ric}(\nabla)$ vanishes identically. 
\end{proof}
\begin{rmk}
Note that in index notation the minimality condition~\eqref{eq:minpde} is equivalent to
\[
\eta_k:=\frac{1}{2}R^{ij}\left(2\nabla_iR_{jk}-\nabla_k R_{ij}\right)=0,
\]
where $R_{ij}$ denotes the Ricci tensor of $\nabla$ and $R^{ij}$ its inverse. We call the $1$-form $\eta$ the~\textit{mean curvature form} of $\nabla$. 
\end{rmk}
\begin{ex}\label{ex:posmetr}
Let $(\Sigma,g)$ be a two-dimensional Riemannian manifold. The Levi-Civita connection $\nabla$ of $g$ has Ricci tensor $\mathrm{Ric}(g)=Kg$, where $K$ denotes the Gauss curvature of $g$. Thus, if $K$ is positive/negative, then $\nabla$ is a timelike/spacelike Lagrangian connection and \cref{thm:minchar} immediately implies that $\nabla$ is minimal. In fact, we will show later (\cref{ppn:class2sphere}) that on the $2$-sphere metrics of positive Gauss curvature are the only examples of minimal Lagrangian connections. 
\end{ex}
Recall that a (non-degenerate) submanifold of a (pseudo-)Riemannian manifold is called~\textit{totally geodesic} if its second fundamental form vanishes identically. We call a timelike/spacelike connection $\nabla$~\textit{totally geodesic} if $o : \Sigma \to (T^*\Sigma,h_{\nabla})$ is a totally geodesic submanifold. We also get:
\begin{cor}
Let $\nabla$ be a timelike/spacelike Lagrangian connection. Then $\nabla$ is totally geodesic if and only if its Ricci tensor is parallel with respect to $\nabla$. 
\end{cor}
\begin{proof}
The second fundamental form vanishes identically if and only if
\[
S_{jki}=S_{ijk}-S_{ikj}.
\]
Since the Ricci tensor is symmetric, the left hand side is symmetric in $j,k$, but the right hand side is anti-symmetric in $j,k$, thus $S_{kji}$ vanishes identically. 
\end{proof}

\subsection{Minimality and the Liouville curvature}

The minimality condition for a timelike/spacelike Lagrangian connection can also be expressed in terms of the Liouville curvature of $\nabla$. To this end we decompose the structure equation\footnote{We define $\eps=(\eps_{ij})$ by $\eps_{ij}=-\eps_{ji}$ with $\eps_{12}=1$ and $\eps^{ij}$ denote the components of the transpose inverse of $\eps$.}
\begin{equation}\label{eq:struceqric1}
\d S_{ij}=S_{ijk}\omega^k+S_{ik}\theta^k_j+S_{kj}\theta^k_i
\end{equation}
into (we compute modulo $\theta^i_j$)
\begin{align*}
\d S_{ij}&=\left(\frac{1}{3}(S_{ijk}+S_{ikj}+S_{jki})+\frac{2}{3}S_{ijk}-\frac{1}{3}(S_{ikj}+S_{jki})\right)\omega_k\\
&=\left(R_{ijk}+\frac{2}{3}L_{(i}\eps_{j)k}\right)\omega^k,
\end{align*}
where we define
\[
R_{ijk}=\frac{1}{3}\left(S_{ijk}+S_{ikj}+S_{jki}\right)
\] and 
\[
L_i=\frac{2}{3}\eps^{jk}\left(S_{ijk}-\frac{1}{2}(S_{ikj}+S_{jki})\right).
\]
\begin{rmk}\label{rmk:lioucurvequi}
The equivariance properties of the function $S=(S_{ij})$ yield $R_a^*L=La \det a$, where we write $L=(L_i)$. Since 
\[
R_a^*\left(\omega^1\wedge\omega^2\right)=(\det a^{-1})\omega^1\wedge\omega^2,
\]
it follows that the there exists a unique $1$-form $\lambda(\nabla)$ on $\Sigma$ taking values in $\Lambda^2(T^*\Sigma)$, such that 
\[
\upsilon^*\lambda(\nabla)=\left(L_1\omega^1+L_2\omega^2\right)\otimes \omega^1\wedge\omega^2.
\]
The $\Lambda^2(T^*\Sigma)$-valued $1$-form was discovered by R.~Liouville and hence we call it the Liouville curvature of $\nabla$. Liouville showed that the vanishing of $\lambda(\nabla)$ is the complete obstruction to $\nabla$ being projectively flat.  
\end{rmk}

Writing $g=\pm \mathrm{Ric}(\nabla)$ for the induced metric, we define $\beta \in \Omega^1(\Sigma)$ by
\[
\beta=\frac{3}{8}\tr_g \mathrm{Sym}\nabla g, 
\]
where $\mathrm{Sym} : \Gamma\left(T^*\Sigma\otimes S^2(T^*\Sigma)\right) \to \Gamma\left(S^3(T^*\Sigma)\right)$ denotes the natural projection. We have:
\begin{ppn}\label{ppn:mincondcon}
A timelike/spacelike Lagrangian connection $\nabla$ on $T\Sigma$ is minimal if and only if 
\begin{equation}\label{eq:mincon}
\lambda(\nabla)=\mp \, 2\star_g\!\beta\otimes dA_g.
\end{equation}
\end{ppn}
\begin{proof}
In order to prove the claim we work on the orthonormal coframe bundle of $g$ which is cut out of the coframe bundle $F$ by the equations $S_{ij}=\pm\delta_{ij}$. By definition, the functions $S_{ijk}$ represent $\nabla\, \mathrm{Ric}(\nabla)$ and hence the functions $R_{ijk}$ represent $\pm\mathrm{Sym}\nabla g$. Therefore, on $F_g$, writing $\upsilon^*\beta=b_i\omega_i$, the components $b_i$ of $\beta$ are
\begin{equation}\label{eq:defbeta}
b_k=\pm\frac{3}{8}\delta^{ij}R_{ijk}.
\end{equation}
Now on $F_g$ the equation~\eqref{eq:mincon} becomes
\[
\left(L_1\omega_1+L_2\omega_2\right)\otimes\omega_1\wedge\omega_2=\mp 2\left(-b_2\omega_1+b_1\omega_2\right)\otimes\omega_1\wedge\omega_2
\]
which is equivalent to 
\[
L_1=\frac{3}{4}(R_{112}+R_{222})\quad \text{and}\quad L_2=-\frac{3}{4}\left(R_{111}+R_{221}\right),
\]
where we have used that $\upsilon^*(\star_g\beta)=-b_2\omega_1+b_1\omega_2$ as well as $\upsilon^*dA_g=\omega_1\wedge\omega_2$ and~\eqref{eq:defbeta}. On the other hand \cref{thm:minchar} implies that the minimality is equivalent to
\begin{align*}
\delta^{ij}\left(2S_{kij}-S_{ijk}\right)&=\delta^{ij}R_{ijk}+\delta^{ij}\left(\frac{4}{3}L_{(k}\eps_{i)j}-\frac{2}{3}L_{(i}\eps_{j)k}\right)\\
&=\delta^{ij}R_{ijk}+\delta^{ij}\left(\frac{2}{3}(L_k\eps_{ij}+L_i\eps_{kj})-\frac{1}{3}(L_i\eps_{jk}+L_{j}\eps_{ik})\right)\\
&=\delta^{ij}R_{ijk}-\frac{4}{3}\delta^{ij}L_i\eps_{jk}=0.
\end{align*}
Written out, this gives the two conditions
\begin{equation}\label{eq:minimality}
R_{111}+R_{221}+\frac{4}{3}L_2=R_{112}+R_{222}-\frac{4}{3}L_1=0,
\end{equation}
which proves the claim.
\end{proof}
\begin{rmk}\label{rmk:liouvcurv}
\cref{ppn:mincondcon} shows that a timelike/spacelike minimal Lagrangian connection is projectively flat if and only if the $1$-form $\beta$ vanishes identically. 
\end{rmk}

\section{Minimal Lagrangian connections}\label{sec:minlag}

We will next compute the structure equations of a timelike/spacelike minimal Lagrangian connection $\nabla$. As before, we work on the orthonormal coframe bundle $F_g$ of the induced metric $g=\pm\mathrm{Ric}(\nabla)$. The submanifold theory discussed in \S 3 tells us that the second fundamental form of $o : \Sigma \to (T^*\Sigma,h_{\nabla})$ is described in terms of a cubic differential and a $(1,\! 0)$-form. Using the definition~\eqref{eq:def10form} of the $(1,\! 0)$-form, the expression~\eqref{eq:2ndfundform} for the second fundamental form and the minimality conditions~\eqref{eq:minimality}, one easily computes that on $F_g$ the $(1,\! 0)$-form is represented by the complex-valued function
\[
b=\pm\frac{3}{8}\left((R_{111}+R_{221})-\i(R_{112}+R_{222})\right).
\]
Hence comparing with~\eqref{eq:defbeta}, we conclude that $b=b_1-\i b_2$. Note that if we define the $(1,\! 0)$-form 
\[
\beta^{1,0}:=\beta+\i\star_g\beta,
\]
then we have $\upsilon^*\beta^{1,0}=(b_1-\i b_2)(\omega_1+\i \omega_2)$, thus the $(1,\! 0)$-form obtained from the normal bundle valued quadratic differential $Q_-$ by using the symplectic form $\Omega_{\nabla}$ is $\beta^{1,0}$. Likewise, $Q_+$ gives the cubic differential $C$ on $\Sigma$ which is represented on $F_g$ by the complex-valued function
\[ 
c=\mp\frac{1}{8}\left(\left(R_{111}-3R_{122}\right)+\i\left(-3R_{112}+R_{222} \right)\right).
\]
From \cref{lem:cubisymtrcor3} and 
\[
\upsilon^*\left(\mathrm{Sym}_0\mathrm{Ric}(\nabla)\right)=\left(R_{ijk}-\frac{3}{2}\delta_{(ij}R_{k)lm}\delta^{lm}\right)\omega_i\otimes\omega_j\otimes \omega_k
\]
we easily compute that the cubic differential $C$ satisfies $\Re(C)=\mp\frac{1}{2}\mathrm{Sym}_0\nabla g$, where the subscript $0$ denotes the trace-free part with respect to $g$. 

The structure equations can now be summarised as follows:
\begin{ppn}\label{ppn:struceqstand}
Let $\Sigma$ be an oriented surface and $\nabla$ a timelike/space\-like minimal Lagrangian connection on $T\Sigma$. Then we obtain a triple $(g,\beta,C)$ on $\Sigma$ consisting of a Riemannian metric $g=\pm\mathrm{Ric}(\nabla)$, a $1$-form $\beta=\frac{3}{8}\tr_g \mathrm{Sym}\nabla g$ and a cubic differential $C$ so that $\mathrm{Re}(C)=\mp\frac{1}{2}\mathrm{Sym}_0\nabla g$. Furthermore, the triple $(g,\beta,C)$ satisfies the following equations
\begin{align}
\label{eq:mainpde}K_g&=\pm 1+2\,|C|_g^2+\delta_g\beta,\\
\label{eq:weakhol}\ov{\partial} C&=\left(\beta-\i\star_g \beta\right)\otimes C,\\
\label{eq:closedness}\d \beta&=0.
\end{align}
\end{ppn}  
\begin{proof}
In our frame adaption where $S_{ij}=\pm\delta_{ij}$ on $F_g$, we obtain from~\eqref{eq:struceqric1}
\[
0=\d S_{ij}=\left(R_{ijk}+\frac{2}{3}L_{(i}\eps_{j)k}\right)\omega_k\pm\delta_{ik}\theta^k_j\pm\delta_{kj}\theta^k_i.
\]
Therefore, writing $\theta_{ij}=\delta_{ik}\theta^k_j$, we have
\begin{equation}\label{eq:conformsubbundle}
\theta_{(ij)}=\mp\frac{1}{2}\left(R_{ijk}+\frac{2}{3}L_{(i}\eps_{j)k}\right)\omega_k.
\end{equation}
For later usage we introduce the notation $c_1=\mp(\frac{1}{8}R_{111}-\frac{3}{8}R_{122})$ and $c_2=\mp(-\frac{3}{8}R_{112}+\frac{1}{8}R_{222})$, so that $c=c_1+\i c_2$. Equation~\eqref{eq:conformsubbundle} written out gives
\begin{align*}
\theta_{11}&=\mp\frac{1}{2}R_{111}\omega_1\mp\left(\frac{3}{4}R_{112}+\frac{1}{4}R_{222}\right)\omega_2,\\
\frac{1}{2}(\theta_{12}+\theta_{21})&=\mp\left(\frac{3}{8}R_{112}-\frac{1}{8}R_{222}\right)\omega_1\mp\left(\frac{3}{8}R_{122}-\frac{1}{8}R_{111}\right)\omega_2,\\
\theta_{22}&=\mp\left(\frac{3}{4}R_{122}+\frac{1}{4}R_{111}\right)\omega_1\mp\frac{1}{2}R_{222}\omega_2.\\
\end{align*}
Defining 
\[
\varphi=\theta_{21}\mp\frac{1}{2}R_{222}\omega_1\pm\left(\frac{1}{4}R_{111}+\frac{3}{4}R_{122}\right)\omega_2,
\] we compute 
\begin{equation}\label{eq:twistconfcon}
\theta=\begin{pmatrix} -\beta & \star_g \beta-\varphi \\ \varphi-\star_g\beta & -\beta\end{pmatrix}+\begin{pmatrix} c_1\omega^1-c_2\omega^2 & -c_2\omega^1-c_1\omega^2\\ -c_2\omega^1-c_1\omega^2 & -c_1\omega^1+c_2 \omega^2\end{pmatrix}.
\end{equation}
The motivation for the definition of $\varphi$ is that we have
\[
\d\omega_1=-\omega_2\wedge\varphi\quad \text{and}\quad \d\omega_2=-\varphi\wedge\omega_1, 
\]
hence $\varphi$ is the Levi-Civita connection form of $g$. In particular, we see that timelike/spacelike minimal Lagrangian connections are twisted Weyl connections. Since $\mathrm{Ric}(\nabla)=\pm g$, it follows that the curvature $2$-form of $\theta$ must satisfy 
\begin{equation}\label{eq:curvcon}
\Theta=\d\theta+\theta\wedge\theta=\begin{pmatrix} 0 & \pm\omega_1\wedge\omega_2 \\ \mp\omega_1\wedge\omega_2 & 0\end{pmatrix}.
\end{equation}
In order to evaluate this condition we first recall that we write $\upsilon^*\beta=b_i\omega_i$ and 
\begin{align*}
\d b_1&=b_{11}\omega_1+b_{12}\omega_2+b_2\varphi,\\
\d b_2&=b_{21}\omega_1+b_{22}\omega_2-b_1\varphi,
\end{align*}
for unique real-valued functions $b_{ij}$ on $F_g$. From~\eqref{eq:twistconfcon} and~\eqref{eq:curvcon} we obtain
\[
\d \beta=-\frac{1}{2}\left(\d\theta_{11}+\d\theta_{22}\right)=\frac{1}{2}\left(\theta_{12}\wedge\theta_{21}+\theta_{21}\wedge\theta_{12}\right)=0
\]
showing that $\beta$ is closed, hence~\eqref{eq:closedness} is verified. Likewise, we also obtain
\begin{align*}
\d \varphi&=\frac{1}{2}(\d\theta_{21}-\d\theta_{12})+\d\! \star_g\!\beta\\
&=(b_{11}+b_{22})\omega_1\wedge\omega_2+\frac{1}{2}\left((\theta_{11}-\theta_{22})\wedge(\theta_{21}+\theta_{12})\right)\mp \omega_1\wedge\omega_2\\
&=-\left(2\left((c_1)^2+(c_2)^2\right)-(b_{11}+b_{22})\pm 1\right)\omega_1\wedge\omega_2.
\end{align*}
Writing $K_g$ for the Gauss curvature of $g$, this last equation is equivalent to
\[
K_g=\pm 1+2\,|C|_g^2+\delta_g\beta,
\]
which verifies~\eqref{eq:mainpde}. 

In order to prove~\eqref{eq:weakhol}, we use 
\[
\upsilon^*\left(\beta-\i\star_g\beta\right)=(b_1+\i b_2)(\omega^1-\i \omega^2).
\]
In light of~\eqref{eq:delbarframe} the condition~\eqref{eq:weakhol} is equivalent to the condition
\begin{equation}\label{eq:weakholcplxnot}
\d c\wedge\omega=\ov{b}c\ov{\omega}\wedge\omega+3\i c\varphi\wedge\omega, 
\end{equation}
where we use the complex notation $b=b_1-\i b_2$, $c=c_1+\i c_2$ and $\omega=\omega_1+\i\omega_2$. Again, from~\eqref{eq:twistconfcon} we compute
\[
c\omega=\frac{1}{2}\left[(\theta_{11}-\theta_{22})-\i\left(\theta_{12}+\theta_{21}\right)\right],
\]
hence 
\begin{multline*}
\d c\wedge\omega=\d(c\omega)-c\d\omega=-\theta_{12}\wedge\theta_{21}\\
+\frac{\i}{2}\left(\theta_{11}\wedge(\theta_{12}-\theta_{21})+\theta_{22}\wedge(\theta_{21}-\theta_{12})\right)-(c_1+\i c_2)(\d \omega_1+\i \d\omega_2).
\end{multline*}
Using~\eqref{eq:twistconfcon} and the structure equations~\eqref{eq:struceqriemcon} this gives
\begin{multline*}
\d c\wedge\omega=3c_2\omega_1\wedge\varphi+3c_1\omega_2\wedge\varphi-2(b_1c_2+b_2c_1)\omega_1\wedge\omega_2\\
+\i\left(-3c_1\omega_1\wedge\varphi+3c_2\omega_2\wedge\varphi+2(b_1c_1-b_2c_2)\omega_1\wedge\omega_2\right),
\end{multline*}
which is equivalent to
\begin{multline*}
\d c\wedge\omega=(b_1+\i b_2)(c_1+\i c_2)(\omega_1-\i\omega_2)\wedge(\omega_1+\i\omega_2)\\
+3\i(c_1+\i c_2)\varphi\wedge(\omega_1+\i\omega_2),
\end{multline*}
that is, equation~\eqref{eq:weakholcplxnot}. This completes the proof. 
\end{proof}
Conversely, unravelling our computations backwards, we also get:
\begin{ppn}
Suppose a triple $(g,\beta,C)$ on an oriented surface $\Sigma$ satisfies the equations \eqref{eq:mainpde},\eqref{eq:weakhol},\eqref{eq:closedness}. Then the connection form~\eqref{eq:twistconfcon} on $F_g$ defines a timelike/spacelike minimal Lagrangian connection $\nabla$ on $T\Sigma$ with $\mathrm{Ric}(\nabla)=\pm g$.   
\end{ppn}
We immediately obtain:
\begin{cor}\label{cor:cortriplecon}
Let $\Sigma$ be an oriented surface. Then there exists a one-to-one correspondence between timelike/spacelike minimal Lagrangian connections on $T\Sigma$ and triples $(g,\beta,C)$ satisfying \eqref{eq:mainpde},\eqref{eq:weakhol},\eqref{eq:closedness}. 
\end{cor}
\begin{proof}
Clearly, the map sending a torsion-free minimal Lagrangian connection $\nabla$ into the set of triples $(g,\beta,C)$ satisfying the above structure equations, is surjective. Now suppose the two triples $(g_1,\beta_1,C_1)$ and $(g_2,\beta_2,C_2)$ on $\Sigma$ satisfy the above structure equations and define the same torsion-free spacelike minimal Lagrangian connection $\nabla$ on $T\Sigma$. Then $g_1=\pm \mathrm{Ric}(\nabla)=g_2$ and consequently we obtain $\beta_1=\beta_2$ as well as $C_1=C_2$, since these quantities are defined in terms of $\nabla \mathrm{Ric}(\nabla)$ by using the metric $g_1=g_2$.  
\end{proof}

\begin{rmk}
\cref{rmk:liouvcurv} immediately implies that a minimal Lagrangian connection is projectively flat if and only if the cubic differential $C$ is holomorphic. 
\end{rmk}
Another consequence of the structure equation is:
\begin{ppn}\label{ppn:chartotgeo}
Let $\nabla$ be a timelike/spacelike Lagrangian connection that is totally geodesic. Then $\nabla$ is the Levi-Civita connection of a metric $g$ of Gauss curvature $K_g=\pm 1$.
\end{ppn}
\begin{proof}
The Lagrangian connection $\nabla$ is totally geodesic if and only if the second fundamental form vanishes identically or equivalently, if $\beta$ and $C$ vanish identically. In this case~\eqref{eq:twistconfcon} implies that $\theta$ is the Levi-Civita connection of $g$ and the structure equation~\eqref{eq:mainpde} gives that $g$ has Gauss curvature $\pm 1$.  
\end{proof}

\begin{rmk}
As we have mentioned previously in \cref{rmk:alignedrep}, a twisted Weyl connection on $(\Sigma,[g])$ defines an AH structure $(\mathfrak{p}(\nabla),[g])$. Moreover, a twisted Weyl connection arising from a triple $(g,\beta,C)$ satisfying $\ov{\partial} C=(\beta-\i\star_g\beta)\otimes C$ defines an associated AH structure $(\mathfrak{p}(\nabla),[g])$ which is~\textit{naive Einstein} in the terminology of~\cite{arXiv:0909.1897,MR3137456,MR3738964}. Therefore, every minimal Lagrangian connection defines a naive Einstein AH structure. 
\end{rmk}

\section{The spherical case}

The system of equations governing minimal Lagrangian connections are easy to analyse on the $2$-sphere $S^2$:

\begin{ppn}\label{ppn:class2sphere}
A connection on the tangent bundle of $S^2$ is minimal Lagrangian if and only if it is the Levi-Civita connection of a metric of positive Gauss curvature.  
\end{ppn}
\begin{proof}
Let $\nabla$ be a minimal Lagrangian connection on $TS^2$ with associated triple $(g,\beta,C)$. Since $\beta$ is closed and $H^1(S^2)=0$, the $1$-form $\beta$ is exact and hence there exists a smooth real-valued function $r$ on $S^2$ such that $\beta=\d r$. Hence we have
$$
\ov{\partial} C=\left(\d r-\i\star_g\d r\right)\otimes C. 
$$
Observe that $\d r-\i\star_g\d r=2\ov{\partial}r$, therefore, the cubic differential $\e^{-2r}C$ is holomorphic. Since, by Riemann-Roch, there are no non-trivial cubic holomorphic differentials on the $2$-sphere, $C$ must vanish identically. The connection form~\eqref{eq:twistconfcon} of $\nabla$ thus becomes
\[
\theta=\begin{pmatrix} -\d r & \star_g \d r-\varphi \\\varphi-\star_g \d r & -\d r\end{pmatrix},
\]
where $\varphi$ denotes the Levi-Civita connection form of $g$. We conclude that $\nabla$ is a Weyl connection given by 
\[
\nabla={}^g\nabla+g\otimes {}^g\nabla r-\d r \otimes \mathrm{Id}-\mathrm{Id}\otimes \d r,
\]
where ${}^g\nabla r$ denotes the gradient of $r$ with respect to $g$. Since the Levi-Civita connection of a Riemannian metric $g$ transforms under conformal change as~\cite[Theorem 1.159]{MR867684}
\[
{}^{\exp(2f)g}\nabla={}^g\nabla-g \otimes {}^g\nabla f+\d f \otimes \mathrm{Id}+\mathrm{Id}\otimes \d f,
\]
we obtain $\nabla={}^{\exp(-2r)g}\nabla$, thus showing that $\nabla$ is the Levi-Civita connection of a Riemannian metric. Moreover, since $\mathrm{Ric}(\nabla)$ must be positive or negative definite, the Gauss curvature of the metric $\e^{-2r}g$ cannot vanish and hence is positive by the Gauss--Bonnet theorem. Finally, \cref{ex:posmetr} shows that conversely the Levi-Civita connection of a Riemannian metric of positive Gauss curvature defines a minimal Lagrangian connection, thus completing the proof.  
\end{proof}

\section{The case of negative Euler-characteristic}

Before we address the classification of minimal Lagrangian connections on compact surfaces of negative Euler characteristic, we observe that every projectively flat spacelike minimal Lagrangian connection defines a properly convex projective structure. Indeed, Labourie gave the following characterisation of properly convex projective manifolds:
\begin{thm}[Labourie~\cite{MR2402597}, Theorem 3.2.1]
Let $(M,\mathfrak{p})$ be an oriented flat projective manifold. Then the following statements are equivalent: 
\begin{itemize}
\item[(i)] $\mathfrak{p}$ is properly convex;
\item[(ii)] there exists a connection $\nabla \in \mathfrak{p}$ preserving a volume form and whose Ricci curvature is negative definite.
\end{itemize} 
\end{thm}
We immediately obtain:
\begin{cor}\label{cor:projflatminlagpropconv}
Let $\nabla$ be a projectively flat spacelike minimal Lagrangian connection on the oriented surface $\Sigma$. Then $\nabla$ defines a properly convex projective structure. 
\end{cor}
\begin{proof}
In \cref{rmk:liouvcurv} we have seen that a minimal Lagrangian connection $\nabla$ is projectively flat if and only if $\beta$ vanishes identically. In the projectively flat case the connection $1$-form $\theta$ of $\nabla$ thus is (see~\eqref{eq:twistconfcon})
\[
\theta=\begin{pmatrix} 0 & -\varphi \\ \varphi & 0\end{pmatrix}+\begin{pmatrix} c_1\omega^1-c_2\omega^2 & -c_2\omega^1-c_1\omega^2\\ -c_2\omega^1-c_1\omega^2 & -c_1\omega^1+c_2 \omega^2\end{pmatrix}.
\]
In particular, the trace of $\theta$ vanishes identically and hence $\nabla$ preserves the volume form of $g$. Since $\mathrm{Ric}(\nabla)=-g$, the claim follows by applying Labourie's result.
\end{proof}
\subsection{Classification}

In \cref{sec:minlag} we have seen that a triple $(g,\beta,C)$ on an oriented surface $\Sigma$ satisfying \eqref{eq:mainpde},\eqref{eq:weakhol},\eqref{eq:closedness} uniquely determines a minimal Lagrangian connection on $T\Sigma$. In this section we will show that in the case where $\Sigma$ is compact and has negative Euler characteristic $\chi(\Sigma)$, the conformal equivalence $[g]$ of $g$ and the cubic differential $C$ also uniquely determine $(g,\beta,C)$ and hence the connection, provided $C$ does not vanish identically. In the case where $C$ does vanish identically the connection is determined uniquely in terms of $[g]$ and $\beta$. 

We start by showing that there are no timelike minimal Lagrangian connections on a compact oriented surface of negative Euler-characteristic (the reader may also compare this with~\cite[Theorem 5.4]{MR3137456}).   
\begin{ppn}
Suppose $\nabla^{\prime}$ is a minimal Lagrangian connection on the compact oriented surface $\Sigma$ satisfying $\chi(\Sigma)<0$. Then $\nabla^{\prime}$ is spacelike.  
\end{ppn}
\begin{proof}
Suppose $\nabla^{\prime}$ were timelike and let $g=\mathrm{Ric}(\nabla^{\prime})$. Then we obtain
\[
\int_{\Sigma}\tr_g\mathrm{Ric}(\nabla^{\prime})dA_g=2\int_{\Sigma}dA_g=2\,\mathrm{Area}(\Sigma,g)\geqslant 0 
\]
and hence \cref{ppn:inftwistedconf} and \cref{thm:minmax} imply that
\[
4\pi\chi(\Sigma)=\sup_{\mathfrak{p} \in \mathfrak{P}(\Sigma)}\inf_{\nabla \in \mathfrak{p}} \int_{\Sigma}\tr_g \mathrm{Ric}(\nabla)dA_g\geqslant 0,
\]
a contradiction. 
\end{proof}
Without loosing generality we henceforth assume that the torsion-free minimal Lagrangian connection $\nabla$ on a compact oriented surface $\Sigma$ with $\chi(\Sigma)<0$ is spacelike. We will show that the triple $(g,\beta,C)$ defined by $\nabla$ is uniquely determined in terms of $[g]$ and $(\beta,C)$. 

Suppose $(g,\beta,C)$ with $\beta$ closed satisfy
\[
K_g=-1+2\,|C|_g^2+\delta_g \beta.
\]
Let $g_0$ denote the hyperbolic metric in $[g]$ and write $g=\e^{2u}g_0$, so that
\[
\e^{-2u}(-1-\Delta_{g_0}u)=-1+2\e^{-6u}|C|_{g_0}^2+\e^{-2u}\delta_{g_0}\beta. 
\]
We obtain
\[
-\Delta_{g_0}u=1+\delta_{g_0}\beta-\e^{2u}+2\e^{-4u}|C|^2_{g_0}.
\]
Omitting henceforth reference to $g_0$ we will show:
\begin{thm}\label{thm:mainpde}
Let $(\Sigma,g_0)$ be a compact hyperbolic Riemann surface. Suppose $\beta \in \Omega^1(\Sigma)$ is closed and $C$ is a cubic differential on $\Sigma$. Then the equation
\begin{equation}\label{eq:mainpdehyp}
-\Delta u=1+\delta\beta-\e^{2u}+2\e^{-4u}|C|^2
\end{equation}
admits a unique solution $u \in C^{\infty}(\Sigma)$.   
\end{thm}
Using the Hodge decomposition theorem it follows from the closedness of $\beta$ that we may write $\beta=\gamma+\d v$ for a real-valued function $v \in C^{\infty}(\Sigma)$ and a unique harmonic $1$-form $\gamma \in \Omega^1(\Sigma)$. Since $\gamma$ is harmonic, it is co-closed, hence~\eqref{eq:mainpdehyp} becomes
\[
\Delta u=-1-\delta\d v+\e^{2u}-2\e^{-4u}|C|^2=-1+\Delta v+\e^{2u}-2\e^{-4u}|C|^2.
\] 
Writing $u^{\prime}:=u-v$, we obtain
\[
\Delta u^{\prime}=-1+e^{2(u^{\prime}+v)}-2\e^{-4(u^{\prime}+v)}|C|^2.
\]
Using the notation $\kappa=-\e^{2v}<0$ and $\tau=\e^{-4v}|C|^2$, as well as renaming $u:=u^{\prime}$, we see that \eqref{thm:mainpde} follows from:

\begin{thm}\label{thm:mainpdevariant}
Let $(\Sigma,g_0)$ be a compact hyperbolic Riemann surface. Suppose $\kappa,\tau \in C^{\infty}(\Sigma)$ satisfy $\kappa<0$ and $\tau \geqslant 0$. Then the equation  
\begin{equation}\label{eq:mainpdevariant}
-\Delta u=1+\kappa\e^{2u}+2\tau\e^{-4u}
\end{equation}
admits a unique solution $u \in C^{\infty}(\Sigma)$. 
\end{thm}
\begin{rmk}
This theorem can also be proved using the technique of sub -- and supersolutions, see~\cite[Chapter 9]{MR3137456}. Here we instead use techniques from the calculus of variations. 
\end{rmk}
In order to prove this theorem we define an appropriate functional $\mathcal{E}_{\kappa,\tau}$ on the Sobolev space $W^{1,2}(\Sigma)$. As usual, we say a function $u \in W^{1,2}(\Sigma)$ is a~\textit{weak solution} of~\eqref{eq:mainpdevariant} if for all $\phi \in C^{\infty}(\Sigma)$
\begin{equation}\label{eq:weakpde}
0=\int_{\Sigma}-\langle \d u,\d \phi\rangle+\left(1+\kappa\e^{2u}+2\tau\e^{-4u}\right)\phi\, dA.
\end{equation}
Note that this definition makes sense. Indeed, it follows from the Moser--Trudinger inequality that the exponential map sends the Sobolev space $W^{1,2}(\Sigma)$ into $L^{p}(\Sigma)$ for every $p< \infty$, hence the right hand side of~\eqref{eq:weakpde} is well defined. 
\begin{lem} Suppose $u \in W^{1,2}(\Sigma)$ is a critical point of the functional 
\[
\mathcal{E}_{\kappa,\tau} : W^{1,2}(\Sigma) \to \ov{\R}, \quad u \mapsto \frac{1}{2}\int_{\Sigma}|\d u|^2-2u-\kappa\e^{2u}+\tau\e^{-4u}dA.
\]
Then $u \in C^{\infty}(\Sigma)$ and $u$ solves~\eqref{eq:mainpdevariant}. 
\end{lem}
\begin{proof}
For $u,v\in W^{1,2}(\Sigma)$ we define $\gamma_{u,v}(t)=u+t v$ for $t \in \R$. We consider the curve $\Gamma_{u,v}=\mathcal{E}_{\kappa,\tau}\circ \gamma_{u,v} : \R \to \R$ so that
\begin{multline}
\Gamma_{u,v}(t)=\frac{1}{2}\int_{\Sigma}|\d u|^2+2t\langle \d u, \d v\rangle+t^2|\d v|^2\\-2(u+t v)-\kappa\e^{2(u+tv)}+\tau \e^{-4(u+tv)}dA.
\end{multline}
The curve $\Gamma_{u,v}(t)$ is differentiable in $t$ with derivative
\[
\frac{\d}{\d t} \Gamma_{u,v}(t)=\int_{\Sigma} \langle \d u,\d v\rangle+t |\d v|^2-v-v\kappa \e^{2(u+tv)}-2v\tau \e^{-4(u+tv)}dA. 
\]
Note that this last expression is well-defined. Again, it follows from the Moser--Trudinger inequality that $\e^{2(u+tv)} \in L^2(\Sigma)$ for all $u,v \in W^{1,2}(\Sigma)$ and  $t \in \R$. Since $W^{1,2}(\Sigma)\subset L^2(\Sigma)$ it follows that $v\e^{2(u+tv)}$ is in $L^1(\Sigma)$ by H\"older's inequality and thus so is $v\e^{-4(u+tv)}$. In particular, assuming that $u$ is a critical point and setting $t=0$ after differentiation gives 
\[
0=\left.\frac{\d}{\d t}\right|_{t=0} \Gamma_{u,v}(t)=\int_{\Sigma}\langle \d u,\d v\rangle-v-v\kappa \e^{2u}-2v\tau \e^{-4u}dA. 
\]
Since $C^{\infty}(\Sigma)\subset W^{1,2}(\Sigma)$ it follows that $u$ is a weak solution of~\eqref{eq:mainpdevariant}. Since the right hand side of~\eqref{eq:mainpdevariant} is in $L^{p}(\Sigma)$ for all $p<\infty$, it follows from the Cald\'eron-Zygmund inequality that $u \in W^{2,p}(\Sigma)$ for any $p<\infty$. Therefore, by the Sobolev embedding theorem, $u$ is an element of the H\"older space $C^{1,\alpha}(\Sigma)$ for any $\alpha<1$. Since the right hand side of~\eqref{eq:mainpdevariant} is H\"older continuous in $u$, it follows from Schauder theory that $u \in C^{2}(\Sigma)$, so that $u$ is a classical solution of~\eqref{eq:mainpdevariant}. Iteration of the Schauder estimates then gives that $u \in C^{\infty}(\Sigma)$.           
\end{proof}
Since $\tau \geqslant 0$ we have $\mathcal{E}_{\kappa,\tau}\geqslant \mathcal{E}_{\kappa,0}$ where here $0$ stands for the zero-function. The functional $\mathcal{E}_{\kappa,0}$ appears in the variational formulation of the equation for prescribed Gauss curvature $\kappa$ of a metric $g=\e^{2u}g_0$ on $\Sigma$. In particular, $\mathcal{E}_{\kappa,0}$ is well-known to be coercive and hence so is $\mathcal{E}_{\kappa,\tau}$. In addition, we have: 
\begin{lem} The functional $\mathcal{E}_{\kappa,\tau}$ is strictly convex on $W^{1,2}(\Sigma)$. 
\end{lem}
\begin{proof}
Let $u,v \in W^{1,2}(\Sigma)$ be given. Using the notation of the previous lemma, we observe that $\Gamma_{u,v}(t)$ is twice differentiable in $t$ with derivative
\begin{equation}\label{eq:secondvar}
\frac{\d^2}{\d t^2}\Gamma_{u,v}(t)=\int_{\Sigma} |\d v|^2-2v^2\kappa\e^{2(u+tv)}+8v^2\tau\e^{-4(u+tv)}dA.
\end{equation}
Note again that by Sobolev embedding $v^2 \in L^2(\Sigma)$ for $v \in W^{1,2}(\Sigma)$ and that both $\e^{2(u+tv)}$ and $\e^{-4(u+tv)}$ are in $L^2(\Sigma)$, hence the right hand side of the equation ~\eqref{eq:secondvar} is well-defined by H\"older's inequality. In particular, computing the second variation gives
\begin{align*}
\mathcal{E}^{\prime\prime}_{\kappa,\tau}(u)[v,v]&=\left.\frac{\d^2}{\d t^2}\right|_{t=0}\mathcal{E}_{\kappa,\tau}(u+tv)\\
&=\int_{\Sigma}|\d v|^2dA+2\int_{\Sigma}v^2(4\tau-\e^{6u}\kappa)\e^{-4u}dA\\
&\geqslant \Vert \d v \Vert^2_{L^2(\Sigma)},
\end{align*}
where we have used that $\tau\geqslant 0$ and $\kappa<0$. Since for a non-zero constant function $v$ we obviously have $\mathcal{E}^{\prime\prime}_{\kappa,\tau}(u)[v,v]>0$ it follows that the quadratic form $\mathcal{E}_{\kappa,\tau}^{\prime\prime}$ is positive definite on $W^{1,2}(\Sigma)$. Hence, the claim is proved.  
\end{proof}
\begin{proof}[Proof of \cref{thm:mainpdevariant}]
We have shown that $\mathcal{E}_{\kappa,\tau}$ is a continuous strictly convex coercive functional on the reflexive Banach space $W^{1,2}(\Sigma)$, hence $\mathcal{E}_{\kappa,\tau}$ attains a unique minimum on $W^{1,2}(\Sigma)$, see for instance~\cite{MR2431434}. Since we know that the minimum is smooth, \cref{thm:mainpdevariant} is proved.
\end{proof}

We define the area of a timelike/spacelike connection to be the area of $o(\Sigma)\subset (T^*\Sigma,h_{\nabla})$. We have:
\begin{thm}\label{thm:areabound}
Let $\nabla$ be a minimal Lagrangian connection on the compact oriented surface $\Sigma$ with $\chi(\Sigma)<0$. Then we have
\[
\mathrm{Area}(\nabla)=- 2\pi\chi(\Sigma)+2\Vert C\Vert_g^2.
\]
\end{thm}
\begin{proof}
We have seen that the Gauss curvature of the metric $o^*h_{\nabla}=g=-\mathrm{Ric}(\nabla)$ defined by a minimal Lagrangian connection $\nabla$ on $\Sigma$ satisfies
\[
K_g=-1+2\,|C|_g^2+\delta_g \beta.
\]
Integrating against $dA_g$ and using the Stokes and Gauss--Bonnet theorem gives
\[
2\pi\chi(\Sigma)=-\mathrm{Area}(\nabla)+2\Vert C\Vert^2_g,
\]
thus proving the claim. 
\end{proof}
\begin{rmk}
An obvious consequence of \cref{thm:areabound} is the area inequality 
\begin{equation}\label{eq:areabound}
\mathrm{Area}(\nabla)\geqslant -2\pi \chi(\Sigma) 
\end{equation}
holding for minimal Lagrangian connections. Recall that if $\nabla$ is projectively flat, then the projective structure defined by $\nabla$ is properly convex. Labourie~\cite{MR2402597} associated to every properly convex projective surface $(\Sigma,\mathfrak{p})$ a unique minimal mapping from the universal cover $\tilde{\Sigma}$ to the symmetric space $\mathrm{SL}(3,\R)/\mathrm{SO}(3)$ which satisfies the very same area inequality, that is~\eqref{eq:areabound}, see~\cite{MR3583351}. Moreover, it is shown in~\cite{MR3583351} that equality holds if and only if $\mathfrak{p}$ is defined by the Levi-Civita connection of a hyperbolic metric. 
\end{rmk}
\begin{defn}
We call a minimal Lagrangian connection $\nabla$ ~\textit{area minimising} if $\nabla$ has area $-2\pi\chi(\Sigma)$. 
\end{defn}
\begin{rmk}
\cref{thm:areabound} shows that a minimal Lagrangian connection $\nabla$ is area minimising if and only if the induced cubic differential vanishes identically. We have shown in \cref{ppn:chartotgeo} that in the projectively flat case -- when $\beta$ vanishes identically --  this translates to $\nabla$ being the Levi-Civita connection of a hyperbolic metric, a statement in agreement with~\cite{MR3583351}.
\end{rmk} 
\cref{thm:mainpde} shows that the triple $(g,\beta,C)$ is uniquely determined in terms of the conformal equivalence class $[g]$, the cubic differential $C$ and the $1$-form $\beta$ on $\Sigma$. Since $C$ can locally be rescaled to be holomorphic, its zeros must be isolated and hence $\beta$ is uniquely determined by $C$ provided $C$ does not vanish identically. Therefore, applying  \cref{cor:cortriplecon} shows:
\begin{thm}
Let $\Sigma$ be a compact oriented surface with $\chi(\Sigma)<0$. Then we have:
\begin{itemize}
\item[(i)] there exists a one-to-one correspondence between area minimising Lagrangian connections on $T\Sigma$ and pairs $([g],\beta)$ consisting of a conformal structure $[g]$ and a closed $1$-form $\beta$ on $\Sigma$;
\item[(ii)] there exists a one-to-one correspondence between non-area minimising minimal Lagrang\-ian connections on $T\Sigma$ and pairs $([g],C)$ consisting of a conformal structure $[g]$ and a non-trivial cubic differential $C$ on $\Sigma$ that satisfies $\ov{\partial} C=\left(\beta-\i\star_g \beta\right)\otimes C$ for some closed $1$-form $\beta$. 
\end{itemize}
\end{thm}

\subsection{Concluding remarks}

\begin{rmk}
We have proved that on a compact oriented surface $\Sigma$ of negative Euler characteristic we have a bijective correspondence between projectively flat spacelike minimal Lagrangian connections and pairs $([g],C)$ consisting of a conformal structure $[g]$ and a holomorphic cubic differential $C$ on $\Sigma$. By the work of Labourie~\cite{MR2402597} and Loftin~\cite{MR1828223}, the latter set is also in bijective correspondence with the properly convex projective structures on $\Sigma$. Since by \cref{cor:projflatminlagpropconv} every projectively flat spacelike minimal Lagrangian connection defines a properly convex projective structure, we conclude that every such projective structure arises from a unique projectively flat spacelike minimal Lagrangian connection. 
\end{rmk}
\begin{rmk}
It is tempting to speculate that a compact oriented projective surface $(\Sigma,\mathfrak{p})$ of negative Euler characteristic contains at most one minimal Lagrangian connection. Some partial results in this direction have been obtained in~\cite{MR3384876,arXiv:1804.04616}.
\end{rmk}

\begin{rmk}
Recall that the universal cover $\tilde{\Sigma}$ of a properly convex projective surface $(\Sigma,\mathfrak{p})$ is a convex subset of $\mathbb{RP}^2$. Pulling back the minimal Lagrangian connection $\nabla \in \mathfrak{p}$ to the universal cover gives a section of $A \to \tilde{\Sigma}$, where now, by the work of Libermann~\cite{MR0066020}, the total space of the affine bundle $A \to \tilde{\Sigma}$ is contained in the submanifold of the para-K\"ahler manifold $A_0\subset \mathbb{RP}^2\otimes \mathbb{RP}^{2*}$ consisting of non-incident point-line pairs. In particular, we obtain a minimal Lagrangian immersion $\tilde{\Sigma} \to A$, recovering the result of Hildebrand~\cite{MR2854275,MR2854277} in the case of two dimensions. Therefore, using the result of Loftin~\cite{MR1828223}, one should be able to show that every properly convex projective manifold arises from a minimal Lagrangian connection.  
\end{rmk}

\begin{rmk}
The case of the $2$-torus can be treated with similar techniques, except for the possible occurrence of Lorentzian minimal Lagrangian connections. This will be addressed elsewhere. 
\end{rmk}

\begin{rmk}
In higher dimensions, the class of totally geodesic Lagrangian connections contains the Levi-Civita connection of Einstein metrics of non-zero scalar curvature. We also refer the reader to~\cite{MR3158041,MR3159950} for a study of Einstein metrics in projective geometry.
\end{rmk}

\begin{rmk}
In~\cite{arXiv:1510.01043}, the author has introduced the notion of an extremal conformal structure for a projective manifold $(M,\mathfrak{p})$. In two-dimensions, the naive Einstein AH structures of~\cite{MR3137456} appear to provide examples of projective surfaces admitting an extremal conformal structure. This may be taken up in future work. 
\end{rmk}

\providecommand{\noopsort}[1]{}
\providecommand{\mr}[1]{\href{http://www.ams.org/mathscinet-getitem?mr=#1}{MR~#1}}
\providecommand{\zbl}[1]{\href{http://www.zentralblatt-math.org/zmath/en/search/?q=an:#1}{zbM~#1}}
\providecommand{\arxiv}[1]{\href{http://www.arxiv.org/abs/#1}{arXiv:#1}}
\providecommand{\doi}[1]{\href{http://dx.doi.org/#1}{DOI}}
\providecommand{\MR}{\relax\ifhmode\unskip\space\fi MR }
\providecommand{\MRhref}[2]{%
  \href{http://www.ams.org/mathscinet-getitem?mr=#1}{#2}
}
\providecommand{\href}[2]{#2}

\end{document}